\documentclass[a4paper,11pt]{article}
\pdfoutput=1
\usepackage[utf8]{inputenc}
\usepackage{kpfonts}
\usepackage{amssymb}
\usepackage{amsthm}
\usepackage{mathtools}
\numberwithin{equation}{section}
\DeclareFontFamily{U}{BOONDOX-calo}{\skewchar\font=45}
\DeclareFontShape{U}{BOONDOX-calo}{m}{n}{
  <-> s*[1.05] BOONDOX-r-calo}{}
\DeclareFontShape{U}{BOONDOX-calo}{b}{n}{
  <-> s*[1.05] BOONDOX-b-calo}{}
\DeclareMathAlphabet{\mathscr}{U}{BOONDOX-calo}{m}{n}

\usepackage{bm}
\usepackage{bbding}
\usepackage{fancyvrb}
\usepackage{graphicx}
\usepackage[square,sort,comma,numbers]{natbib}
\usepackage[nottoc]{tocbibind}
\usepackage{paralist}
\usepackage[usenames,dvipsnames]{xcolor}
\usepackage{color}
\usepackage{soul}
\usepackage[hyperindex,pageanchor,raiselinks,bookmarks,pdftex,unicode]{hyperref}
\hypersetup{breaklinks=true}
\hypersetup{unicode}
\makeatletter
\def\@makechapterhead#1{
  {\parindent \z@ \raggedright \normalfont
   \Huge\bfseries \thechapter. #1
   \par\nobreak
   \vskip 20\p@
}}
\def\@makeschapterhead#1{
  {\parindent \z@ \raggedright \normalfont
   \Huge\bfseries #1
   \par\nobreak
   \vskip 20\p@
}}

\makeatother

\theoremstyle{plain}
\newtheorem{thm}{Theorem}[section]
\newtheorem{lemma}[thm]{Lemma}

\newtheorem{prop}[thm]{Proposition}

\theoremstyle{plain}

\theoremstyle{remark}

\newtheorem{rem}{Remark}[section]
\newtheorem{example}{Example}[section]

\newcommand{\R}{\mathbb{R}}

\newcommand{\N}{\mathbb{N}}
\newcommand{\Z}{\mathbb{Z}}

\newcommand{\T}[1]{#1^\top}

\newcommand{\goto}{\rightarrow}

\newcommand\eqdef{\mathrel{\overset{\makebox[0pt]{\mbox{\normalfont\tiny\sffamily def}}}{=}}}

\newcommand\restr[2]{{
  \left.\kern-\nulldelimiterspace 
  #1 
  \vphantom{\big|} 
  \right|_{#2} 
  }}
	
\newcommand{\mres}{%
  \,\raisebox{-.127ex}{\reflectbox{\rotatebox[origin=br]{-90}{$\lnot$}}}\,%
}

\DeclareMathOperator*{\argmin}{\arg\!\min}

\DeclareMathOperator*{\esssup}{ess\,sup\,}

\def\Xint#1{\mathchoice
   {\XXint\displaystyle\textstyle{#1}}%
   {\XXint\textstyle\scriptstyle{#1}}%
   {\XXint\scriptstyle\scriptscriptstyle{#1}}%
   {\XXint\scriptscriptstyle\scriptscriptstyle{#1}}%
   \!\int}
\def\XXint#1#2#3{{\setbox0=\hbox{$#1{#2#3}{\int}$}
     \vcenter{\hbox{$#2#3$}}\kern-.5\wd0}}
\def\dashint{\Xint-}

\newcommand{\tc}[1]{\textcolor{black}{#1}}
\newcommand{\tg}[1]{\textcolor{black}{#1}}
\newcommand{\tb}[1]{\textcolor{black}{#1}}

\def\zf{\mathfrak{z}}
\def\zfb{\tilde{\mathfrak{z}}}

\DeclareMathOperator*{\d2d}{\bar{\nabla}^{\rm 2d}}

\newcommand{\PH}[1]{P\text{-}H^#1(0,L;\R^3)}
\newcommand{\PHgen}[2]{P\text{-}H^#1(0,L;\R^{#2})}

\def\yb{\tilde{y}}
\def\ybk{\yb^{(k)}}

\def\xb{x}
\def\Qb{\tilde{Q}}

\def\Lkex{\Lambda_k^{\rm ext}}
\def\Lkexc{\Lambda_k'^{,\mathrm{ext}}}
\def\Os{\Omega^{\rm ext}}
\def\Lkexb{\tilde{\Lambda}_k^{\rm ext}}
\def\Lkexcb{\tilde{\Lambda}_k'^{,\mathrm{ext}}}

\newcommand{\yp}[1]{\smash{\overset{+}{\mathscr{y}}}^{(#1)}}
\newcommand{\yw}[1]{\smash{\overset{\leftharpoonup}{\mathscr{y}}}^{(#1)}}
\newcommand{\Rp}[2]{\smash{\overset{+}{R}}_{#1}^{(#2)}}

\newcommand{\ypp}[2]{\smash{\overset{+}{y}}_{#1}^{(#2)}}

\newcommand{\md}{\textup{d}}
\newcommand{\e}{\varepsilon}
\newcommand{\wto}{\rightharpoonup}
\renewcommand{\a}{\alpha}
\renewcommand{\b}{\beta}

\newcommand{\g}{\gamma}

\newcommand{\ph}{\varphi}
\renewcommand{\O}{{\Omega}}

\newcommand{\Id}{\mathrm{Id}}

\newcommand{\calE}{\mathcal{E}}
\newcommand{\calC}{\mathcal{C}}

\newcommand{\calA}{\mathcal{A}}

\newcommand\pl{\partial}

\begin{document}
\title{A continuum model for brittle nanowires derived from an atomistic description by $\Gamma$-convergence}
\author{Bernd Schmidt \\ \textsc{\footnotesize{Institut f\"{u}r Mathematik, Universit\"{a}t Augsburg, D-86135 Augsburg, Germany}} \\ \footnotesize{Email address: \href{mailto:bernd.schmidt@math.uni-augsburg.de}{\texttt{bernd.schmidt@math.uni-augsburg.de}}} \\ 
   \and Ji\v{r}\'{\i} Zeman \\ \textsc{\footnotesize{Institut f\"{u}r Mathematik, Universit\"{a}t Augsburg, D-86135 Augsburg, Germany}} \\ \footnotesize{Email address: \href{mailto:geozem@seznam.cz}{\texttt{geozem@seznam.cz}}}}
\date{\today}
\maketitle
\begin{abstract}
Starting from a particle system with short-range interactions, we derive a continuum model for the bending, torsion, and brittle fracture of inextensible rods moving in three-dimensional space. As the number of particles tends to infinity, it is assumed that the rod's thickness is of the same order as the interatomic distance. Fracture energy in the \mbox{$\Gamma$-limit} is expressed by an implicit cell formula, which covers different modes of fracture, including (complete) cracks, folds, and torsional cracks. In special cases, the cell formula can be significantly simplified. Our approach applies e.g.\ to atomistic systems with Lennard-Jones-type potentials and is motivated by the research of ceramic nanowires.
\smallskip 

\smallskip

\end{abstract}
\section{Introduction}

Ceramic and semiconductor nanowires (composed of Si, SiC, $\text{Si}_3\text{N}_4$, $\text{TiO}_2$, or ZnO etc.) under loading exhibit large deflections, but also brittle or ductile fracture. \cite{mechNW} Their mechanical behaviour is often very different from that of bulk materials, size- and structure-dependent, and  influenced by surface energy. Laboratory testing at the nanoscale still poses various challenges, so modelling and simulation play an important role in the advancement of nanotechnology. \cite{nanoMod} 

To set off on a path towards elastic-fractural modelling of nanowires, in this article we derive from three-dimensional atomistic models a continuum theory for ultrathin rods whose elastic energy is of the order corresponding to bending or torsion. After treating the purely elastic case in \cite{elRods}, here we extend our model considerably by adding liability of the material to develop brittle cracks.

Our work stands at the crossroads of three paths of research in applied analysis which are:
\begin{enumerate}
	\item[(DR)] rigorous derivation of elasticity theories for thin structures (often referred to as \textit{dimension reduction});
	\item[(D-C)] discrete-to-continuum limits;
	\item[(F)] fracture mechanics.
\end{enumerate}

An important tool in all these three branches is $\Gamma$-convergence. \cite{BraiBeg,BraiHand}

In (DR) the aim is to understand the relation between three-dimensional elasticity theory and effective theories for lower-dimensional bodies, such as plates, rods or beams. \cite{Ciarlet2, Antman, OReilly} With the pioneering contributions of L. Euler and D. Bernoulli, the journey started more than two centuries before the first nanowires were manufactured. Yet, most mathematically rigorous derivations of such theories first appeared no sooner than in the 1990s. \cite{strings, membranes, Anze} A decade later, the famous discovery of a quantitative rigidity estimate in \cite{FrM02} brought forth an abundance of works on bending theories. \cite{FrM02,FrM06,MMh4}

As for (D-C), `establishing the status of elasticity theory with respect to atomistic models' was listed by Ball among outstanding open problems in elasticity. \cite{openPbs} Research has been devoted to studying the Cauchy--Born rule \cite{valFail,EMing}, pointwise limits of interaction energies \cite{blanc} and their $\Gamma$-limits \cite{AC,BS09,BrS13}, or to finding atomistic deformations approximating a given solution of the equations of elasticity \cite{OT13,BrS16,Br17}.  See also \cite{blanc2} for a survey. 

The interest of mathematicians in (F) was particularly ignited after Francfort and Marigo \cite{FM98} elaborated on the influential model by Griffith, using modern variational methods (see e.g. \cite{F20yrs,BFM} for further references). In variational models of fracture, be it \textit{brittle} or \textit{cohesive} \cite{Barenblatt}, we typically find functionals involving the sum of elastic and fracture energy:
\begin{equation}\label{eq:fracE}
\int_\O W(\nabla y(x))\md x+\int_{J_y}\kappa(y^+(x)-y^-(x),\nu(x))\md\mathcal{H}^{d-1}(x).
\end{equation}
In the above, $W\colon\R^{3\times 3}\goto[0,\infty)$ stands for the stored energy density of a material body $\O\subset\R^d$, $d\in\{2,3\}$, $y^+-y^-$ is the jump of the deformation $y:\O\goto\R^d$ across the crack set $J_y$, $\nu$ denotes the normal vector field to $J_y$, and $\kappa\colon(\R^d)^2\goto[0,\infty]$ is the fracture toughness.

Given the myriads of physical situations that emerge in modern materials science, it seems natural that researchers have made efforts to bridge some of the gaps between (DR), (D-C) and (F). 

Combining (DR) and (D-C) is motivated by the need of accurate models for thin structures in nanoengineering, such as thin films or nanotubes. \cite{FJ00,BS08,BS08prop,ABC08} Interestingly, when the thickness $h$ of the reference crystalline body is very small (i.e. comparable to the interatomic distance $\e$), the simultaneous $\Gamma$-limit as $\e\goto 0+$, $h\goto 0+$ gives rise to new \textit{ultrathin plate} or \textit{rod theories} which could not be obtained by (DR) in the purely continuum setting. \cite{BS06,BrS19,elRods}

Atomistic effects also lie at the core of crack formation and propagation. \cite{Hud20,atomFrac} However, up to now combinations of (D-C) and (F) have only been explored in specific situations such as one-dimensional chains of atoms \cite{braiCica,scaSchlo,jansen}, scalar-valued models \cite{BGlong}, or cleavage in crystals \cite{FrdS14,FrdS15,FrdS15b}. 

Similarly, despite the recent progress, theories uniting (DR) and (F) are still under development. In linearized elasticity, models for brittle plates \cite{BabHenao,AlmiTasso,FPZ10,baldelli}, beams \cite{Ginster} or shells \cite{fracShells} have been derived mostly using a weak formulation in $SBD$ or $GSBD$ function spaces \cite{BD,GBD}. The nonlinear setting of membranes \cite{BraiFonseca,babadjian,membNotPenet}, on the other hand, employs the more regular spaces $SBV$ and $GSBV$. \cite{SBVbook} As for nonlinear bending theories, the lack of a piecewise quantitative rigidity estimate in 3D presents an obstacle, so the result of \cite{BS17} with a dimension reduction from 2D to 1D seems rather isolated; we also refer to \cite{FrVoids,BZK} for materials with voids.

\begin{figure}[h]\label{fig:NW}
  \caption{Fracture of a thin rod composed of atoms.}
  \centering
\includegraphics[width=7.5cm]{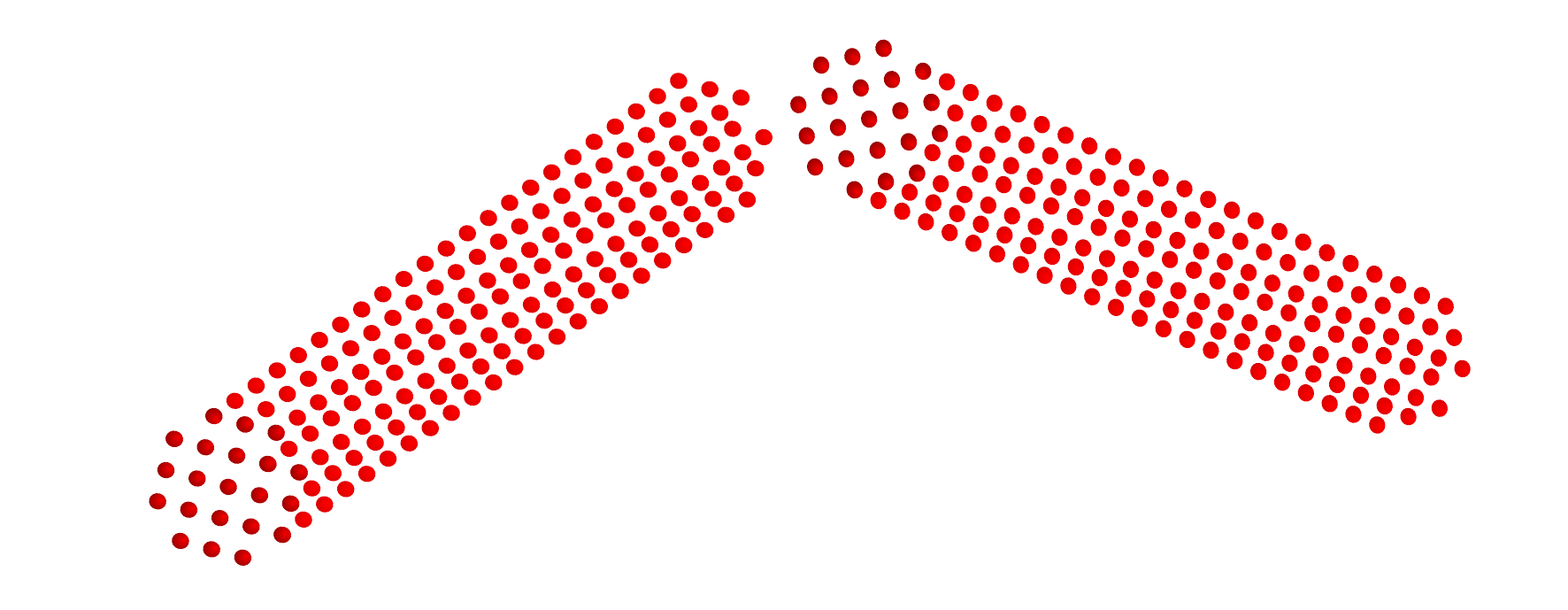}
\end{figure}

In this article, we treat a problem that falls into all three branches (DR), (D-C) and (F). Our main Theorem \ref{Gamma} provides the $\Gamma$-limit of atomic interaction energies defined on cubic crystalline lattices in the shape of a slender rod. Unlike in the purely elastic model from \cite{elRods}, we now replace the interaction potentials (expressed by a cell energy function $W_{\rm cell}$ like in e.g. \cite{valFail,Conti,BS06}) with a sequence $(W_{\rm cell}^{(k)})_{k=1}^\infty$ of cell energies to ensure that elastic deformations (bending and torsion) are comparably favourable in terms of energy as cracks \tc{(see Figure \ref{fig:NW} for an illustration)}. This is specifically expressed in condition \tc{(W5)} for the constants $(\bar{c}_1^{(k)})_{k=1}^\infty$, which give a lower bound on the cost of placing atoms far away from each other (see Subsection \ref{sec:E}). Physically we can interpret this as considering a sequence of materials that are mutually similar but are characterized by different values of material parameters. The limiting strain energy has, just like in \eqref{eq:fracE}:
\begin{enumerate}
\item A bulk part that coincides with its counterpart in \cite{elRods} and features an \textit{ultrathin correction} and \textit{atomic surface layer terms}, neither of which appears in the corresponding rod theory \cite{MMh4} derived by (DR) without (D-C). These traits might make a model better-suited for the description of nanostructures.
\item A fracture part which turns out to be a weighted sum over the singular set of a limiting deformation. The weights are given by an implicit cell formula $\ph=\ph(y^+-y^-,(R^-)^{-1}R^+)$, where $y^+-y^-\in\R^3$ denotes the jump of the deformation mapping at a specified crack point and $(R^-)^{-1}R^+\in{\rm SO}(3)$ is related to kinks/folds or torsional rupture.
\end{enumerate}

Implicit cell formulas arise in $\Gamma$-convergence problems in homogenization \cite{BraiHand} or phase transitions \cite{2Dsolid,surfE2D,phasTra}.

To comment on some important aspects of the proofs, in the \textit{liminf inequality} we first derive a preliminary cell formula by a blowup technique reminiscent of \cite{blowup,SBVbook} and then relate it to a more simple asymptotic formula which uses rigid boundary values (cf. \cite{polycrystals}). The atomistic setting allows us to circumvent the unavailability of a 3D piecewise rigidity theorem in $SBV$ (in fact, it is enough to work with piecewise Sobolev functions here).  The main challenge of our analysis is, however, to provide a matching \textit{limsup inequality}. Due to the $k$-dependency of the interaction potential $W_{\rm cell}^{(k)}$, it is a priori not clear how to construct a global recovery sequence $(y^{(k)})$ that not only works for a specific subsequence. We resolve this difficulty by establishing a localization of cracks on the atomic length scale, which appears to be of some independent interest. More precisely, we argue that an approximative minimizing sequence $(\mathscr{y}^{(k)})$ for $\ph$ can be chosen with cracks confined to a fixed number of atomic slices (Lemma \ref{rigCrack}), which lets us transfer $\mathscr{y}^{(k)}$ to a lattice with different interatomic distances (Proposition \ref{arbAtomDist}) and thus define $(y^{(k)})$ for every $k\in\N$. $\Gamma$-convergence problems involving brittle fracture often have to deal with pieces of the deformed body escaping to $\infty$. As our limiting theory is one-dimensional we can sidestep working on $GSBV$-type spaces and instead obtain a limiting functional on piecewise $H^2$ functions. By an explicit construction using assumption \tc{(W9)} in Lemma \ref{LinftyBdd} we show that $L^\infty$ (or weaker) bounds could be imposed energetically so as to ensure matching compactness properties of low-energy sequences.

After describing our discrete model in Section 2, we prove a compactness theorem for sequences of bounded energy in Section 3. The lower bound in the $\Gamma$-convergence result from Section 4 is shown in Section 5 and then followed in Section 6 by an analysis of the cell formula and the construction of recovery sequences for Theorem \ref{Gamma}(ii). Section 7 provides examples of interatomic potentials to which our approach applies. In Section 8, we show that for full cracks and a class of mass-spring models there is an explicit expression for the cell formula. Moreover, it is proved that in such models, the energy needed to produce a full crack is strictly greater than the energy of a mere kink. The last short discussion section gives some hints on possible future research. 

\textbf{Notation.} We write $\mathrm{dist}(B_1,B_2):=\inf\{|x^{(1)}-x^{(2)}|;\;x^{(1)}\in B_1,\; x^{(2)}\in B_2\}$ for $B_1,B_2\subset\R^3$. Whenever the symbol $\pm$ appears in an equation, we mean that the equation holds both in the version with $+$ in all occurrences \textit{ and } in the version with $-$. The letter $C$ denotes a positive generic constant, whose value may be different in different instances. One-sided limits are written as $f(\sigma\pm)=\lim_{x\goto\sigma\pm} f(x)$. Further, $\R_{\rm skew}^{3\times 3}=\{A\in\R^{3\times 3};\;A=-\T{A}\}$. The symbol $A_{\cdot i}$ denotes the $i$-th column of a matrix $A$ and $\mathcal{H}^n$ is the $n$-dimensional Hausdorff measure. The restriction $\mu\mres K$ of a measure $\mu$ to the measurable set $K$ is defined by $\mu\mres K(U)=\mu(U\cap K)$.

\section{Model assumptions and preliminaries}\label{sec:pas}

\subsection{Atomic lattice and discrete gradients}
In our particle interaction model, $\Lambda_k=([0,L]\times \frac{1}{k}\bar{S})\cap\frac{1}{k}\mathbb{Z}^3$, $k\in\N$, is a cubic atomic lattice -- the reference configuration of a thin rod of length $L>0$. The interatomic distance $1/k$ is directly proportional to the thickness of the rod.

The rod's cross section is represented with a bounded domain $\emptyset\neq S\subset\R^2$. We assume that there is a set $\mathcal{L}'\subset(\frac{1}{2}+\Z)^2$ such that
$$S=\mathrm{Int}\,\bigcup_{\xb'\in\mathcal{L}'}\bigl(\xb'+\Bigl[-\frac{1}{2},\frac{1}{2}\Bigr]^2\bigr).$$
Moreover, should it happen that $\xb'+\{-\frac{1}{2},\frac{1}{2}\}\subset\mathcal{L}:=\bar{S}\cap \Z^2$, it is assumed that $\xb'\in\mathcal{L}'$. The symbol $\Lambda_k'$ is used for the lattice of midpoints of open lattice cubes with sidelength $1/k$ and corners in $\Lambda_k$.

Our lattice $\Lambda_k$ undergoes a static deformation $y^{(k)}\colon \Lambda_k\goto\R^3$. The main aim of this paper is to investigate the asymptotic behaviour as $k$ becomes large and to establish an effective continuum model as $k\goto+\infty$.

Sometimes it will be advantageous to work with a rescaled lattice that has unit distances between neighbouring atoms. The points of this lattice are written with hats over their coordinates, i.e.\ if $x=(x_1, x_2, x_3)\in\Lambda_k$ we introduce $\hat{x}_1:=kx_1$, $\hat{x}'=(\hat{x}_2,\hat{x}_3):=kx'=k(x_2,x_3)$ and $\hat{y}^{(k)}(\hat{x}_1,\hat{x}_2,\hat{x}_3):=k y^{(k)}(\frac{1}{k}\hat{x}_1,\frac{1}{k}\hat{x}')$ so that $\hat{y}^{(k)}\colon k\Lambda_k\to \R^3$. Then $\hat{\Lambda}_k$ and $\hat{\Lambda}'_k$ denote the sets of all $\hat{x}=(\hat{x}_1,\hat{x}_2,\hat{x}_3)$ such that the corresponding downscaled points $x$ are elements of the sets $\Lambda_k$ and $\Lambda_k'$, respectively. We will frequently use these eight direction vectors $\zf^1,\dots,\zf^8$:

\begin{center}
\begin{tabular}{l l}
$\zf^1 = \frac{1}{2}\T{(-1,-1,-1)}$, & $\zf^5 = \frac{1}{2}\T{(+1,-1,-1)},$\\
$\zf^2 = \frac{1}{2}\T{(-1,-1,+1)}$, & $\zf^6 = \frac{1}{2}\T{(+1,-1,+1)},$\\
$\zf^3 = \frac{1}{2}\T{(-1,+1,+1)}$, & $\zf^7 = \frac{1}{2}\T{(+1,+1,+1)},$\\
$\zf^4 = \frac{1}{2}\T{(-1,+1,-1)}$, & $\zf^8 = \frac{1}{2}\T{(+1,+1,-1)}.$
\end{tabular}
\end{center}

With these vectors we can describe the deformation of a unit cell $\hat{x}+\{-\frac{1}{2},\frac{1}{2}\}^3$ centred at $\hat{x}\in\hat{\Lambda}'_k$ -- let $\vec{y}^{\,(k)}(\hat{x})=(\hat{y}^{(k)}(\hat{x}+\zf^1)|\cdots|\hat{y}^{(k)}(\hat{x}+\zf^8))\in\R^{3\times 8}$. Further we introduce $\langle \hat{y}^{(k)}(\hat{x})\rangle=\frac{1}{8}\sum_{i=1}^8 \hat{y}^{(k)}(\hat{x}+\zf^i)$ and the discrete gradient $\bar{\nabla}\hat{y}^{(k)}(\hat{x})=(\hat{y}^{(k)}(\hat{x}+\zf^1)-\langle \hat{y}^{(k)}(\hat{x})\rangle|\cdots|\hat{y}^{(k)}(\hat{x}+\zf^8)-\langle \hat{y}^{(k)}(\hat{x})\rangle)\in\R^{3\times 8}$. A discrete gradient has the sum of columns equal to $0$ and an important special case is the matrix $\bar{\Id}:=(\zf^1|\cdots|\zf^8)\in\R^{3\times 8}$, which satisfies $\bar{\Id}=\bar{\nabla}\mathrm{id}$. Further we define two noteworthy subsets of $\R^{3\times 8}$, later used for characterizing rigid motions:
$$\bar{\rm SO}(3):=\{R\,\bar{\Id};\;R\in\mathrm{SO}(3)\},\quad V_0:=\{(c|\cdots|c)\in\R^{3\times 8};\;c\in\R^3\}.$$

\subsection{Rescaling, interpolation and extension of deformations}\label{sec:rie}
To handle sequences of deformations defined on a common domain $\O=(0,L)\times S$, we set $\ybk(\xb_1,\xb_2,\xb_3):=y^{(k)}(\xb_1,\frac{1}{k}\xb')$ for $(x_1,\frac{1}{k}\xb')\in\Lambda_k$ and interpolate $\ybk$ as follows so that it is defined even outside lattice points.

Write $\zfb^i:=(\frac{1}{k}\zf_1^i,\zf_2^i,\zf_3^i)$ and $\Qb(\bar{\xb})=\bar{\xb}+(-\frac{1}{2k},\frac{1}{2k})\times (-\frac{1}{2},\frac{1}{2})^2$ for $\bar{\xb}\in\tilde{\Lambda}_k'=\{\xi\in\O;\;(k\xi_1,\xi')\in\hat{\Lambda}_k'\}$. First, we set $\ybk(\bar{\xb}):=\frac{1}{8}\sum_{i=1}^8\ybk(\bar{\xb}+\zfb^i)$ and for each face $\tilde{F}$ of the block $\Qb(\bar{\xb})$ and the corresponding centre $\xb_{\tilde{F}}$ of the face $\tilde{F}$, define $\ybk(\xb_{\tilde{F}}):=\frac{1}{4}\sum_j\ybk(\bar{\xb}+\zfb^j)$, where the sum runs over all $j$ such that $\bar{\xb}+\zfb^j$ is a corner of $\tilde{F}$. Now we interpolate $\ybk$ in an affine way on every simplex $\tilde{T}=\mathrm{conv}\{\bar{\xb},\bar{\xb}+\zfb^i,\bar{\xb}+\zfb^j,\xb_{\tilde{F}}\}$, where $|\zf^i-\zf^j|=1$ and $\bar{\xb}+\zfb^i,\bar{\xb}+\zfb^j\in\tilde{F}$ (there are 24 simplices within $\Qb(\bar{\xb})$). Like this, $\ybk$ is differentiable almost everywhere, so we can define $\nabla_k\ybk:=\bigl(\frac{\partial \ybk}{\partial \xb_1}\,|\,k\frac{\partial \ybk}{\partial \xb_2}\,|\,k\frac{\partial \ybk}{\partial \xb_3}\bigr)$. For any face $\tilde{F}$ of $\Qb(\bar{\xb})$ with face centre $\xb_{\tilde{F}}$, the piecewise affine interpolation satisfies
\begin{equation}\label{eq:surfVolMean}
\ybk(\xb_{\tilde{F}})=\dashint_{\tilde{F}}\,\ybk\md \mathcal{H}^2\text{ and }\ybk(\bar{\xb})=\dashint_{\Qb(\bar{\xb})}\ybk(\xi)\md\xi.
\end{equation}

We also set $\bar{\nabla}_k\ybk(\bar{\xb}):=k(\ybk(\bar{\xb}_1+\frac{1}{k}\zf_1^i,\bar{\xb}'+(\zf^i)')-\sum_{j=1}^8\ybk(\bar{\xb}_1+\frac{1}{k}\zf_1^j,\bar{\xb}'+(\zf^j)'))_{i=1}^8$.

For the following reasons we now extend deformations to certain auxiliary surface lattices: 
\begin{itemize}
	\item surface energy needs to be modelled;
	\item in part we would like to apply $\Gamma$-convergence results from \cite{elRods};
	\item a fixed domain on which the convergence of $(\ybk)$ is formulated sometimes does not match with its inscribed crystalline lattice (specifically in the $\xb_1$-direction). 
\end{itemize}
We present here the necessary tools, without too much emphasis on this technical issue later, referring to \cite[Subsection~2.3]{elRods} for more details and a proof, adapted from \cite{BS09}. Consider a portion $(a,b)\times S\subset (0,L)\times S$ of the rod. Let $a_k=\frac{1}{k}\lceil ka\rceil$, $b_k=\frac{1}{k}\lfloor kb\rfloor$, and
\begin{align*}
\mathcal{L}^{\rm ext}
&=\mathcal{L}+\{-1,0,1\}^2,
& 
\Lkex
&=\{a_k-\tfrac{1}{k},a_k,\ldots,b_k+\tfrac{1}{k}\}\times \tfrac{1}{k}\mathcal{L}^{\rm ext},
\\
\mathcal{L}'^{,\rm ext}
&=\mathcal{L}'+\{-1,0,1\}^2,
&
\Lkexc
&=\{a_k-\tfrac{1}{2k},a_k+\tfrac{1}{2k},\ldots,b_k+\tfrac{1}{2k}\}\times \tfrac{1}{k}\mathcal{L}'^{,\rm ext},\\ 
S^{\rm ext}
&=S+(-1,1)^2,
&
\Omega^{\rm ext}_k
&= (a_k-\tfrac{1}{k}, b_k+\tfrac{1}{k}) \times S^{\rm ext},\\
\Lkexb &=\{a_k-\tfrac{1}{k},a_k,\ldots,b_k+\tfrac{1}{k}\}\times\mathcal{L}^{\rm ext},&\Lkexcb &=\{a_k-\tfrac{1}{2k},a_k+\tfrac{1}{2k},\ldots,b_k+\tfrac{1}{2k}\}\times \mathcal{L}'^{,\rm ext}.
\end{align*}
\begin{lemma}\label{lemma:ext-general}
There are extensions $y^{(k)}\colon \Lkex \to \R^3$ such that their interpolations $\ybk$ satisfy
\begin{align*}
\esssup_{\O^{\rm ext}_k} \,\mathrm{dist}^2(\nabla_k \ybk,\mathrm{SO}(3)) 
&\leq C \esssup_{(a_k,b_k)\times S} \,\mathrm{dist}^2(\nabla_k \ybk,\mathrm{SO}(3)) 
\shortintertext{and}
\int_{\O^{\rm ext}_k} \mathrm{dist}^2(\nabla_k \ybk,\mathrm{SO}(3)) \md x 
&\leq C \int_{(a_k,b_k)\times S} \mathrm{dist}^2(\nabla_k \ybk,\mathrm{SO}(3)) \md x. 
\end{align*}
\end{lemma}
For $\xb\in \overline{\O_k^{\rm ext}}$, we denote by $\bar{\xb}$ an element of $\Lkexcb$ that is closest to $\xb$. In what follows we always understand the symbols $\Lkex$, $\Lkexc$ etc. with $a:=0$ and $b:=L$, unless stated otherwise. We also set $\Os:=(0,L)\times S^{\rm ext}$.

\subsection{Energy}\label{sec:E}
Let $L_k=\frac{1}{k}\lfloor kL\rfloor$, $\hat{\Lambda}_k'^{,\mathrm{surf}} = \{\frac{1}{2},\ldots,kL_k-\frac{1}{2}\}\times (\mathcal{L}'^{,\rm ext} \setminus \mathcal{L}')$, and $\hat{\Lambda}_k'^{,\mathrm{end}} = \{-\frac{1}{2},kL_k+\frac{1}{2}\} \times \mathcal{L}'^{,\rm ext}$. We give this definition of strain energy $E^{(k)}$:
\begin{equation}\label{eq:Ek}
\begin{aligned}
E^{(k)}(y^{(k)})=&\sum_{\hat{x}\in\hat{\Lambda}_k'} W_{\rm cell}^{(k)}\left(\vec{y}^{\,(k)}(\hat{x})\right)+\sum_{\hat{x}\in\hat{\Lambda}_k'^{,\mathrm{surf}}} W_{\rm surf}^{(k)}\left(\hat{x}',\vec{y}^{\,(k)}(\hat{x})\right)\\
&+\sum_{\hat{x}\in\hat{\Lambda}_k'^{,\mathrm{end}}} W_{\rm end}^{(k)}\left(\frac{1}{k}\hat{x}_1,\hat{x}',\vec{y}^{\,(k)}(\hat{x})\right)
\end{aligned}
\end{equation}
with $W_{\rm cell}^{(k)}\colon \R^{3\times 8}\to [0,\infty]$, $W_{\rm surf}^{(k)}\colon (\mathcal{L}'^{,\rm ext}\setminus\mathcal{L}')\times\R^{3\times 8}\goto [0,\infty]$ and $W_{\rm end}^{(k)}\colon \{-\frac{1}{2k},L_k+\frac{1}{2k}\}\times\mathcal{L}'^{,\rm ext}\times\R^{3\times 8}\goto [0,\infty]$. The terms with $W_{\rm surf}^{(k)}$ and $W_{\rm end}^{(k)}$ are useful for incorporating surface energy (see \cite{elRods} for further clarification). For convenience we assume that for every $\vec{y}\in\R^{3\times 8}$, $W_{\rm surf}^{(k)}(\cdot,\vec{y})$ is extended to a piecewise constant function on $S^{\rm ext}\setminus \bar{S}$ which is equal to $W_{\rm surf}^{(k)}(\hat{x}',\vec{y})$ on $\hat{x}'+(-\frac{1}{2},\frac{1}{2})^2$. Sometimes it will be \tc{useful} to group the terms, so for $\vec{y}\in\R^{3\times 8}$ we set
\begin{equation*}
\tc{W_{\rm tot}^{(k)}(\hat{x}',\vec{y})}=\begin{cases}
W_{\rm cell}^{(k)}(\vec{y}) & \tc{\hat{x}'\in \bar{S}},\\
W_{\rm surf}^{(k)}(\hat{x}',\vec{y}) & \hat{x}'\in (S^{\rm ext}\setminus \bar{S}).
\end{cases}
\end{equation*}
In our $\Gamma$-convergence statement, we consider the rescaled energy $\frac{1/k^3}{1/k^4}E^{(k)}=kE^{(k)}$, where $k^3$ is the order of the number of particles per unit volume in a bulk system and $1/k^4$ is the appropriate power of a rod's thickness for studying the \textit{bending/torsion energy regime} (see e.g. \cite{MMh6} for more context).

\noindent\textbf{Assumptions} on the cell energy functions $W_{\rm cell}^{(k)}$, $W_{\rm surf}^{(k)}$, and $W_{\rm end}^{(k)}$. 

Hereafter $\mathscr{W}^{(k)}$ stands for $W_{\rm cell}^{(k)}$, $W_{\rm surf}^{(k)}(\hat{x}',\cdot)$ with $\hat{x}'\in \mathcal{L}'^{,\rm ext} \setminus \mathcal{L}'$, and for $W_{\rm end}^{(k)}(\tg{\frac{1}{k}\hat{x}_1},\hat{x}',\cdot)$ with $\hat{x}\in \hat{\Lambda}_k'^{,\mathrm{end}}$.
\begin{enumerate}
\item[(W1)] Frame-indifference: $\mathscr{W}^{(k)}(R\vec{y}+(c|\cdots|c))=\mathscr{W}^{(k)}(\vec{y})$ for all $R\in {\rm SO}(3)$, $\vec{y}\in \R^{3\times 8}$, $c\in \R^3$, and $k\in\N$.
\item[(W2)] Energy well: For every $k\in\N$, $\mathscr{W}^{(k)}$ attains a minimum (equal to 0) at rigid deformations, i.e.\ deformations $\vec{y}=(\hat{y}_1|\cdots|\hat{y}_8)$ with $\hat{y}_i=R\zf^i+c$ for all $i\in\{1,\dots,8\}$ and some $R\in\mathrm{SO}(3)$, $c\in\R^3$.
\item[(W3)] Independence of $k$ in the elastic regime: There are parameters $c_{\rm frac}^{(k)}\searrow 0$ such that $\lim_{k\goto\infty}k(c_{\rm frac}^{(k)})^2\in(0,\infty)$ and an elastic stored energy function $W_0\colon\mathcal{L}'^{,\rm ext}\times\R^{3\times 8}\goto[0,\infty]$ such that we have $\forall\, k\in\N\;\forall\,\vec{y}\in\R^{3\times 8}\;\forall \tc{\xb'\in \mathcal{L}'^{,\rm ext}}$:
\begin{equation*}
W_{\rm tot}^{(k)}(\xb',\vec{y})=W_0(\xb',\vec{y})\quad\text{if }\;\mathrm{dist}(\bar{\nabla}\hat{y},\bar{\rm SO}(3))\leq c_{\rm frac}^{(k)}.
\end{equation*}
Further, there exists a $C>0$ independent of $k\in\N$ such that
\begin{equation*}
W_{\rm end}^{(k)}(\tfrac{1}{k}\hat{x}_1,\hat{x}',\vec{y})\leq C\mathrm{dist}^2(\bar{\nabla}\hat{y},\bar{\rm SO}(3))\quad\text{for any }\hat{x}\in \hat{\Lambda}_k'^{,\mathrm{end}},
\end{equation*}
$\vec{y}=(\hat{y}_1|\cdots|\hat{y}_8)\in\R^{3\times 8}$, and $\bar{\nabla}\hat{y}=\vec{y}-(\sum_{j=1}^8\hat{y}_j)(1,\dots,1)$ with $\mathrm{dist}(\bar{\nabla}\hat{y},\bar{\rm SO}(3))\leq c_{\rm frac}^{(k)}$.
\item[(W4)] Regularity in $k$:
$W_{\rm tot}^{(k+1)}(\xb',\vec{y})\geq \frac{k}{k+1}W_{\rm tot}^{(k)}(\xb',\vec{y})$ for all $ k\in\N\;\forall\,\vec{y}\in\R^{3\times 8}\;\forall \tc{\xb'\in \mathcal{L}'^{,\rm ext}}.$
\item[(W5)] Non-degeneracy in the elastic and the fracture regime: The function $W_0|_{{\mathcal{L}'\times\R^{3\times 8}}}$ is independent of $\xb'$ (hence we omit it from the notation in this region) and satisfies
\begin{align*}
W_0(\vec{y})\geq c_{\rm W}\mathrm{dist}^2(\bar{\nabla}\hat{y},\bar{\rm SO}(3))\quad\forall\,\vec{y}\in\R^{3\times 8}
\end{align*}
for a constant $c_{\rm W}>0$. Writing $W_{\rm cell}^{(k)}(\vec{y})=\bar{W}^{(k)}(\vec{y})$ if $\mathrm{dist}(\bar{\nabla}\hat{y},\bar{\rm SO}(3))> c_{\rm frac}^{(k)}$, we assume that the mappings $\bar{W}^{(k)}$ can be chosen such that
\begin{align*}
\bar{W}^{(k)}(\vec{y}) \geq \bar{c}_1^{(k)} \quad  \forall\, k\in\N\;\forall\,\vec{y}\in\R^{3\times 8}
\end{align*}
for a sequence $(\bar{c}_1^{(k)})_{k=1}^\infty$ of positive numbers with $\lim_{k\goto\infty}k\bar{c}_1^{(k)}\in (0,\infty)$. 
\item[(W6)] $\mathscr{W}^{(k)}$ is everywhere Borel measurable and $W_0(\hat{x}',\cdot)$, $\hat{x}'\in\mathcal{L}^{',\mathrm{ext}}$, is of class $\mathcal{C}^2$ in a neighbourhood of $\bar{\rm SO}(3)$.
\item[(W7)] If $i\in\{1,2,\dots,8\}$, $\hat{x}'\in\mathcal{L}'^{,\mathrm{ext}}\setminus\mathcal{L}'$, and $\vec{y}=(\hat{y}_1|\cdots|\hat{y}_8)$, then $\vec{y}\mapsto W_{\rm surf}^{(k)}(\hat{x}',\vec{y})$ may depend on $\hat{y}_i$ only if $\hat{x}'+(\zf^i)'\in\mathcal{L}$. If $\xb_1\in\{-\frac{1}{2k},L_k+\frac{1}{2k}\}$, then $\vec{y}\mapsto W_{\rm end}^{(k)}(\xb_1,\hat{x}',\vec{y})$ may depend on $\hat{y}_i$ only if $(\xb_1,\hat{x}')+\zfb^i\in\tilde{\Lambda}_k$.
\end{enumerate}
The quadratic form associated with $\nabla^2 W_{\rm surf}^{(k)}(\bar{\Id})$ is denoted by $Q_{\rm surf}$.

Throughout we will assume that Assumptions \tc{(W1)--(W7)} are satisfied. We also introduce conditions which imply that long-range interactions of atoms are bounded or even are negligible.
\begin{enumerate}
\item[(W8)] We say that inelastic interactions are \textit{bounded} if
\begin{align*}
\mathscr{W}^{(k)}(\vec{y}) \leq \bar{C}_1^{(k)} \quad  \forall\, k\in\N\;\forall\,\vec{y}\in\R^{3\times 8}
\end{align*}
for a sequence $(\bar{C}_1^{(k)})_{k=1}^\infty$ of positive numbers with $\lim_{k\goto\infty}k\bar{C}_1^{(k)}\in (0,\infty)$. 
\item[(W9)] We say that the cell energies have \textit{maximum interaction range} scaling with $(M_k)_{k=1}^{\infty}$, where $M_k\to0$, $M_kk \to \infty$, if the following holds true: If there is a partition $\{1,\ldots,8\}=J_1\dot{\cup}J_2\dot{\cup}\cdots\dot{\cup}J_{n_{\rm C}}$  such that for some $\vec{y},{\vec{y}}\,'\in\R^{3\times 8}$ one has 
\begin{equation*}
 \min_{1\leq \ell<m\leq n_{\rm C}}\mathrm{dist}(\{\hat{y}_{i_\ell}\}_{i_\ell\in J_\ell},\{\hat{y}_{i_m}\}_{i_m\in J_m})\geq M_kk 
 \quad\text{and}\quad 
 \min_{1\leq \ell<m\leq n_{\rm C}}\mathrm{dist}(\{\hat{y}'_{i_\ell}\}_{i_\ell\in J_\ell},\{\hat{y}'_{i_m}\}_{i_m\in J_m})\geq M_kk
\end{equation*}
and there are rigid motions given by $R_m\in \mathrm{SO}(3)$ and $c_m\in\R^3$ such that 
\begin{equation*}
 \hat{y}'_{i_m} = R_m\hat{y}_{i_m} + c_m 
 \quad \forall\, i_m\in J_m,\; m=1,\ldots, n_{\rm C},  
\end{equation*}
then 
\begin{equation*}
 |\mathscr{W}^{(k)}({\vec{y}}\,') - \mathscr{W}^{(k)}(\vec{y})| \le \frac{\tc{C_{\rm far}}}{M_k\tc{k^2}}
\end{equation*}
for a uniform constant $C_{\rm far}>0$.
\end{enumerate}

\begin{rem}
We remark that the assumption in \tc{(W4)} is a monotonicity assumption only for $kW_{\rm tot}^{(k)}(\xb',\cdot)$ but not for $W_{\rm tot}^{(k)}(\xb',\cdot)$ itself. It is in line with our assuming that the elastic energy is independent of $k$ in \tc{(W3)} and the fracture toughness scales with $\frac{1}{k}$, cf.\ \tc{(W5)}. 
\end{rem}

\begin{rem}
 By (W2), \tc{(W3)}, and \tc{(W6)} we have 
 \begin{align*}
\mathscr{W}^{(k)}(\vec{y}) \le c_{\rm w}\mathrm{dist}^2(\bar{\nabla}\hat{y},\bar{\rm SO}(3)) 
\end{align*}
for a constant $c_{\rm w}$ and all $\vec{y}\in\R^{3\times 8}$ such that $\mathrm{dist}(\bar{\nabla}\hat{y},\bar{\rm SO}(3))\leq c_{\rm frac}^{(k)}$. Moreover, by (W2), \tc{(W5)} and \tc{(W6)} the quadratic form $Q_3$ associated with $\nabla^2 W_0(\bar{\Id})$, is positive definite on $\mathrm{span}\{V_0\cup\R_{\rm skew}^{3\times 3}\bar{\Id}\}^\bot$.
\end{rem}

\subsection{Piecewise Sobolev functions}

We work with the linear spaces $\PHgen{m}{\ell}$, $m=1,2$, \tc{$\ell\in\N$}, of functions that are piecewise Sobolev in the following sense:
\begin{align}
\PHgen{m}{\ell}
&:=\bigl\{\yb\in L^1((0,L);\R^\ell);\; 
\exists\text{ partition } (\sigma^i)_{i=0}^{n+1} \text{ of }[0,L]\nonumber \\ 
&\qquad\qquad\forall i\in\{1,2,\dots,n+1\}\colon\yb|_{(\sigma_{i-1},\sigma_i)}\in H^m((\sigma_{i-1},\sigma_i);\R^\ell)\bigr\}.\label{eq:yPcwHm}
\end{align}
Here we say that $(\sigma^i)_{i=0}^{n+1}$ is a partition of $[0,L]$ if $0=\sigma^0 < \sigma^1<\cdots< \sigma^{n+1}=L$. Suppose $\yb\in\PHgen{m}{\ell}$ and $\{\sigma^i\}_{i=0}^{n+1}$ is the minimal set with property \eqref{eq:yPcwHm}. For $m=1$ one has 
\begin{align*}
S_{\yb}
:=\{\sigma\in(0,L);\;\yb(\sigma-)\ne\yb(\sigma+)\} 
=\{\sigma^i\}_{i=0}^{n+1}.
\end{align*}
For $m=2$ we have 
\begin{align*}
S_{\yb'}:=\{\sigma\in\{\sigma^i\}_{i=1}^{n};\;\yb(\sigma-)=\yb(\sigma+)\},\qquad 
S_{\yb}:=\{\sigma^i\}_{i=1}^{n}\setminus S_{\yb'},
\end{align*}
where the set $S_{\yb}$ is the \textit{jump set} of $\yb$ and $S_{\yb'}$ the jump set of the derivative $\pl_{\xb_1}\yb$.

\section{Compactness}

\begin{thm}\label{comp}
Suppose the sequence $(y^{(k)})_{k=1}^\infty$ of lattice deformations fulfils
\begin{equation}\label{eq:limsEfin}
\limsup_{k\goto\infty}\bigl(kE^{(k)}(y^{(k)})+\tg{|| y^{(k)}||_{\ell^\infty(\Lambda_k;\R^3)}\bigr)}<+\infty
\end{equation}
Then after applying the extension scheme from Subsection \ref{sec:rie} we can find an increasing sequence $(k_j)_{j=1}^\infty\subset\N$, functions $\yb\in \PH{2}$, $d_2,d_3\in \PH{1}$ with $R=(\pl_{\xb_1}\yb|d_2|d_3)\in\mathrm{SO}(3)$ a.e., and a partition $(\sigma^i)_{i=0}^{\bar{n}_{\rm f}+1}$ of $[0,L]$ such that for any $\eta\in(0,\frac{1}{2}\min_{0\leq i\leq \bar{n}_f} |\sigma^{i+1}-\sigma^i|)$ and every $0\leq i\leq \bar{n}_{\rm f}$ we have:
\begin{enumerate}[(i)]
\item $\yb^{(k_j)}\goto \yb$ in $L^2(\O^{\rm ext};\R^{3\times 3})$; 
\item $\nabla_{k_j}\yb^{(k_j)}\goto R=(\pl_{\xb_1}\yb|d_2|d_3)$ in $L^2((\sigma^i+\eta,\sigma^{i+1}-\eta)\times S^{\rm ext};\R^{3\times 3})$;
\item $\tc{\mathrm{dist}(\bar{\nabla}_{k_j}\yb^{(k_j)},\bar{\mathrm{SO}}(3))} \le c_{\rm frac}^{(k)}$ on $(\sigma^i+\eta,\sigma^{i+1}-\eta)\times S^{\rm ext}$, for $j$ sufficiently large;
\item if we define the measures $\mu_k$ on $[0,L]$ by
$$\mu_k(A) = \sum_{\substack{\hat{x}\in\hat{\Lambda}_k'^{,\mathrm{ext}},\\ \hat{x}_1\in kA}}kW_{\rm tot}^{(k)}\bigl(\hat{x}',\vec{y}^{\,(k)}(\hat{x})\bigr),$$
for Borel sets $A$, then $\mu_{k_j}\wto^*\mu$ for a Radon measure $\mu$.
\end{enumerate}
\end{thm}
\begin{proof}
\tc{By properties of the extension scheme from Subsection \ref{sec:rie} (see \cite[Remark~2.1]{elRods}) there is a constant $\hat{C}_{e}\geq 1$} such that for any $\xb\in\Lkexcb$, setting $\mathcal{U}(\xb)=\bigl(\{x_1-\frac{1}{k},x_1,x_1+\frac{1}{k}\}\times\mathcal{L}'\bigr) \tc{\cap \tilde{\Lambda}_k'}$ we have 
\begin{equation}\label{eq:nbhd-rig-est}
\mathrm{dist}^2(\tc{\bar{\nabla}_k}\ybk(\xb),\bar{\rm SO}(3))\leq \hat{C}_{e}^2\sum_{\xi\in\mathcal{U}(\xb)}\mathrm{dist}^2(\tc{\bar{\nabla}_k}\ybk(\xi),\bar{\rm SO}(3)).
\end{equation}
Let $S_k(\xb_1)$ denote a slice of the rod at the point $\xb_1$:
$$ S_k(\xb_1)=\bigl(\frac{1}{k}\lfloor k\xb_1\rfloor,\frac{1}{k}\lfloor k\xb_1\rfloor+\frac{1}{k}\bigr)\times S^{\rm ext},\quad \xb_1\in [0,L].$$
A slice $S_k(\xb_1)$ is regarded as broken if there is an $\xb'\in S$ such that
$$\mathrm{dist}\bigl(\bar{\nabla}\hat{y}^{(k)}(k\xb_1, \xb'), \bar{\rm SO}(3)\bigr)> \frac{c^{(k)}_{\rm frac}}{\sqrt{3\sharp\mathcal{L}'}\hat{C}_{\rm e}}.$$
Like this, for any $x$ such that the slice $S_k(\xb_1)$ and, if existent, the neighbouring slices $S_k(\xb_1\pm\frac{1}{k})$ are not broken, $\bar{\nabla}_k\ybk(\xb)$ is at most $c_{\rm frac}^{(k)}$-far from $\bar{\rm SO}(3)$ even if \tc{$\xb\in \O_k^{\rm ext}\setminus (0,L_k)\times S^{\rm ext}$}. Write $X_1^{(k)}$ for the set of all midpoints of the $\xb_1$-projections of broken slices:
$$X_1^{(k)}=\bigl\{\xb_1\in\bigl(\frac{1}{2k}+\frac{1}{k}\Z\bigr)\cap [0,L);\;S_k(\xb_1)\text{ is broken}\bigr\}.$$
We have $\sharp X_1^{(k)}\leq C_{\rm f}$ with $C_{\rm f}>0$ independent of $k$, since by Assumptions \tc{(W3) and (W5)}
\begin{align*}
\min \Big\{ W_{\rm cell}^{(k)}(\vec y);\; \vec y\in\R^{3\times 8},\ \mathrm{dist}(\bar{\nabla}\hat{y}, \bar{\rm SO}(3)) \ge \frac{c^{(k)}_{\rm frac}}{\sqrt{3\sharp\mathcal{L}'}\hat{C}_{\rm e}} \Big\}
\ge \min \Big\{ \frac{c_W (c^{(k)}_{\rm frac})^2}{3\sharp\mathcal{L}'\hat{C}_{\rm e}^2},\ \bar{c}_1^{(k)} \Big\} 
\ge \frac{c}{k}
\end{align*}
for a constant $c > 0$ and so 
\begin{align}
C \ge kE^{(k)}(y^{(k)})
&\ge\sum_{\hat{x}\in\hat{\Lambda}_k'^{,\rm ext}}kW_{\rm tot}^{(k)}\bigl(\hat{x}',\vec{y}^{\,(k)}(\hat{x})\bigr)
\label{eq:limEbdd} \\
&\ge c \sharp X_1^{(k)} 
+\underbrace{k\sum_{\tc{\hat{x}\in\hat{\Lambda}_k'^{,\rm ext}},\ \hat{x}_1 \notin kX_1^{(k)} }W_{\rm tot}^{(k)}\bigl(\hat{x}',\vec{y}^{\,(k)}(\hat{x})\bigr)}_{\text{elastic part }(\ge 0) }.  \label{eq:E-intact} 
\end{align}
If we pass to a subsequence $\{k_j\}_{j=1}^\infty\subset\N$, we find $n_{\rm f}\in \N$, $0\leq n_{\rm f}\leq C/c$, such that for every $j\in\N$, there are always precisely $n_{\rm f}$ broken slices, i.e.\ $\forall j\in\N\colon \sharp X_1^{(k_j)}=n_{\rm f}$, and 
$$X_{1}^{(k_j)}=\{s_j^1,s_j^2,\dots, s_j^{n_{\rm f}}\},\quad s_j^1<s_j^2<\cdots< s_j^{n_{\rm f}}.$$
We observe that the location $s_j^i$ of the $i$-th broken slice, $1\leq i\leq n_{\rm f}$, remains in the compact interval $[0, L]$, so we construct a further subsequence, which we still denote by $(k_j)_{j=1}^\infty$, so that
$$\forall i\in\{1,2,\dots, n_{\rm f}\}\colon \lim_{j\goto\infty}s_j^i=s^i\in[0,L].$$
Naturally it can be that some of the limiting positions of cracks $s^i$, $i=1,2,\dots n_{\rm f}$, coincide or appear at the endpoints of the rod, hence we rewrite
$$X_1:=\{s^i;\; 0<s^i<L,\;1\leq i\leq n_{\rm f}\}=\{\sigma^i\}_{i=1}^{\bar{n}_{\rm f}},$$
where the number $\bar{n}_{\rm f}\leq n_{\rm f}$. Further, $\sigma^0:=0$ and $\sigma^{\bar{n}_{\rm f}+1}:=L$.

Suppose $0<\eta<\frac{1}{2}\min_{0\leq i\leq \bar{n}_f} |\sigma^{i+1}-\sigma^i|$. If $j$ is large enough, then for all $i$, $0\leq i\leq \bar{n}_{\rm f}$, 
$$[\sigma^i+\eta,\sigma^{i+1}-\eta]\cap \bigl(\xb_1-\frac{3}{2\tc{k_j}},\xb_1+\frac{3}{2\tc{k_j}}\bigr)=\emptyset.$$ 
Thus the regions $[\sigma^i+\eta,\sigma^{i+1}-\eta]\times S$ are intact, so we can replace $W_{\rm cell}^{(k)}$ by $W_0$ and safely apply our results about purely elastic rods here (see \cite[Theorem~2.4]{elRods}). Specifically, $\yb^{(k_j)} \goto \yb$ in $L^2((\sigma^i+\eta,\sigma^{i+1}-\eta)\times S^{\rm ext};\R^3)$, $\nabla_{k_j}\yb^{(k_j)}\goto R=(\pl_{\xb_1}\yb|d_2|d_3)$ in $L^2((\sigma^i+\eta,\sigma^{i+1}-\eta)\times S^{\rm ext};\R^{3\times 3})$, and the $\xb'$-independent limit satisfies $\yb\in H^2((\sigma^i+\eta,\sigma^{i+1}-\eta);\R^3)$, $d_2,d_3\in H^1((\sigma^i+\eta,\sigma^{i+1}-\eta);\R^3)$, and $R\in\mathrm{SO}(3)$ a.e. (We extracted another subsequence without changing the subindices.)  By passing to a diagonal sequence we find a single sequence that \tc{satisfies convergence properties (i)--(ii)} for any choice of $\eta$. Moreover, the $L^{\infty}$ bound in \eqref{eq:limsEfin} and the uniform energy bound in \eqref{eq:E-intact} show that indeed $\yb \in \PH{2}$ and $R \in \PHgen{1}{3\times 3}$. Finally passing to yet another subsequence (not relabelled), we find $\mu_{k_j}\wto^*\mu$ for some Radon measure $\mu$ since \eqref{eq:limEbdd} implies $\sup_k \mu_k([0,L])< \infty$.
\end{proof}

\section{Main result}
\begin{thm}\label{Gamma}
If $k\goto\infty$, we have $E^{(k)}\stackrel{\Gamma}{\goto}E_{\rm lim}$, more precisely:
\begin{enumerate}
\item[(i)] \upshape{(liminf inequality)} \itshape Let $(y^{(k)})_{k=1}^\infty$ be a sequence of lattice deformations such that their piecewise affine interpolations and extensions $(\ybk)_{k=1}^\infty\subset H^1(\O_k^{\rm ext};\R^3)$, defined in Subsection \ref{sec:rie}, converge in $L^2(\O^{\rm ext};\R^3)$ to $\yb\in L^2((0,L);\R^3)$ for which there is a partition $(\varsigma^i)_{i=0}^{\tilde{n}_{\rm f}+1}$ of $[0,L]$ such that $\yb|_{(\varsigma^i,\varsigma^{i+1})}\in H^1((\varsigma^i,\varsigma^{i+1})\times S^{\rm ext};\R^3)$, $0\leq i\leq \tilde{n}_{\rm f}$.

Assume further that for any $\eta>0$ sufficiently small, we have $k\pl_{\xb_s}\ybk\goto d_s\in L^2((0,L);\R^3)$ in $L^2((\varsigma^i+\eta,\varsigma^{i+1}-\eta)\times S^{\rm ext};\R^3)$, $s=2,3$, $0\leq i\leq \tilde{n}_{\rm f}$ ($L_{\rm loc}^2$-convergence). Then
$$E_{\rm lim}(\yb,d_2,d_3)\leq \liminf_{k\goto\infty} kE^{(k)}(y^{(k)}).$$

\item[(ii)] \upshape{(existence of a recovery sequence)} \itshape Let $\yb\in L^2((0,L);\R^3)$ be such there is a partition $(\varsigma^i)_{i=0}^{\tilde{n}_{\rm f}+1}$ of $[0,L]$ for which $\yb|_{(\varsigma^i,\varsigma^{i+1})}\in H^1((\varsigma^i,\varsigma^{i+1});\R^3)$, and let $d_2,d_3\in L^2((0,L);\R^3)$. Then there exists a sequence of lattice deformations $(y^{(k)})_{k=1}^\infty$ such that their piecewise affine interpolations and extensions $(\ybk)_{k=1}^\infty\subset H^1(\tc{\O_k^{\rm ext}};\R^3)$ satisfy $\ybk\goto\yb$ in $L^2(\tc{\O^{\rm ext}};\R^3)$, $k\frac{\pl\ybk}{\pl\xb_s}\goto d_s$ in $L_{\rm loc}^2((\varsigma^i,\varsigma^{i+1})\times S^{\rm ext};\R^3)$ for $s=2,3$, $0\leq i\leq \tilde{n}_{\rm f}$, and
$$\lim_{k\goto\infty} kE^{(k)}(y^{(k)})=E_{\rm lim}(\yb,d_2,d_3).$$
Moreover, if $||\yb||_{\tc{L^\infty((0,L);\R^3)}} \le M$ and the cell energies satisfy the maximum interaction range property \tc{(W9)}, then for any \tc{$(\zeta_k)_{k=1}^{\infty}\subset(0,1)$} with $\zeta_k \searrow 0$ and $\zeta_k/M_k \to \infty$ one can choose $y^{(k)}$ such that $\tg{||y^{(k)}||_{\tc{\ell^\infty(\Lambda_k;\R^3)}}} \le M+\zeta_k$.
\end{enumerate}
The limit energy functional is given by
\begin{equation*}
E_{\rm lim}(\yb,d_2,d_3)=
\begin{cases}
\begin{aligned}
\frac{1}{2}\int_0^L &Q_3^{\rm rel}(\T{R}\pl_{\xb_1}R)\md\xb_1\\
&+\sum_{\sigma\in S_{\yb}\cup S_R}\varphi\bigl(\yb(\sigma+)-\yb(\sigma-),(R(\sigma-))^{-1}R(\sigma+)\bigr)\end{aligned} & \text{if } (\yb,d_2,d_3)\in\calA,\\
+\infty & \text{otherwise},
\end{cases}
\end{equation*}
where $R:=(\pl_{\xb_1}\yb|d_2|d_3)$, $S_R:=S_{\yb'}\cup S_{d_2}\cup S_{d_3}$, and the class of admissible deformations
\begin{multline*}
\calA:=\bigl\{(\yb,d_2,d_3)\in (L^1(\O;\R^3))^3;\;\yb,d_2,d_3\textit{ do not depend on }\xb_2,\xb_3,\\
(\yb,d_2,d_3)\in \PH{2}\times(\PH{1})^2\text{ as functions of }\xb_1\text{ only,}\\
\Bigl(\frac{\pl\yb}{\pl\xb_1}\,\Big|\,d_2\,\Big|\, d_3\Bigr)\in{\rm SO}(3)\text{ a.e. in }(0,L)\bigr\}.
\end{multline*}
The relaxed quadratic form $Q_3^{\rm rel}\colon\R^{3\times 3}_{\rm skew}\goto[0,+\infty)$ is defined as
\begin{multline}\label{eq:minProb}
Q_3^{\rm rel}(A):=\min_{\substack{\a\colon\mathcal{L}^{\rm ext}\goto\R^3\\g\in\R^3}}\sum_{\xb'\in\mathcal{L}'^{,\rm ext}} Q_{\rm tot}\Bigl(\xb',\frac{1}{2}\Bigl(A\begin{pmatrix}
0\\
\xb_2\\
\xb_3
\end{pmatrix}+g\Bigr)(-1,-1,-1,-1,1,1,1,1)\\
+\frac{1}{4}A\begin{pmatrix}
0&0&0&0&0&0&0&0\\
1&1&-1&-1&-1&-1&1&1\\
1&-1&-1&1&-1&1&1&-1
\end{pmatrix}+(\d2d\a|\d2d\a)\Bigr)
\end{multline}
with $Q_{\rm tot}(\xb',\cdot)=Q_3+Q_{\rm surf}(\xb',\cdot)$, and $\ph\colon \R^3 \times \mathrm{SO}(3) \to [0,\infty]$ is introduced in \eqref{eq:phi}.
\end{thm}
\begin{rem}
It follows from the positive semidefiniteness of $Q_{\rm tot}$ that the minimum in \eqref{eq:minProb} is attained.
\end{rem}
\begin{rem}
The elastic part of our limiting functional includes a matrix expressing what we call an \textit{ultrathin correction} -- it is the first term on the second line of \eqref{eq:minProb}. The term is responsible for atomic effects that a continuum theory merely based on the Cauchy–Born rule would not capture.
\end{rem}
\begin{rem}\label{phBdd}
Assumptions \tc{(W3)}, \tc{(W5)} and the compactness result \cite[Theorem~2.4]{elRods} in the elastic case imply that $\varphi \ge \bar{c}_1$ for some constant $\bar{c}_1 > 0$ on $\R^3 \times \mathrm{SO}(3) \setminus \{(0,\tc{\Id})\}$ (and $\varphi(0,\tc{\Id}) = 0$). If \tc{(W8)} holds true, then we also have $\varphi \le \bar{C}_1$ for a constant $\bar{C}_1 < \infty$.
\end{rem}
\begin{rem}
The universality of the sequence $\zeta_k$ obtained in (ii) would allow to impose an $L^{\infty}$ constraint energetically by simply setting $E^{(k)}(y^{(k)}) = +\infty$ if $||y^{(k)}||_{\infty} > \tg{M+\zeta_k}$. One then has a directly matching compactness result in Theorem~\ref{comp}.
\end{rem}
\begin{rem}
The convergence of deformations used in Theorem \ref{Gamma} is equivalent to
\begin{align*}
\ybk(\cdot,x')\goto\yb&\text{ in }L^2((0,L);\R^3) \text{ for every } x'\in\mathcal{L} \text{ and} \\  \bar{\nabla}_k\ybk\goto R\,\bar{\Id}&\text{ in }L_{\rm loc}^2((\varsigma^i,\varsigma^{i+1})\times S;\R^{3\times 8}) \text{ for } 0\leq i\leq \tilde{n}_{\rm f},
\end{align*}
which shows the limit's independence of our interpolation scheme.
\end{rem}

\section{Proof of the lower bound}
The proof is divided into four parts.
\subsection{First step -- elastic part}
Since the conclusion is immediate if the liminf is infinite, let us assume the contrary; $\ybk\goto\yb$ in $L^2(\O;\R^3)$ and after extracting a subsequence,
\begin{equation}\label{eq:liminfEbdd} 
\lim_{k\goto\infty} kE^{(k)}(y^{(k)})=\liminf_{k\goto\infty} kE^{(k)}(y^{(k)})< \infty.
\end{equation}
Let $(\sigma^i)_{i=0}^{\bar{n}_{\rm f}+1}$, $\nabla_{k_j}\yb^{(k_j)}$, $\mu_k$, $\mu$ be as in Theorem \ref{comp} and fix $\eta>0$ small. Then by the results about purely elastic rods (\cite[Theorem~3.1]{elRods}), the bound 
\begin{equation*}
\liminf_{k\goto\infty}\sum_{\substack{\hat{x}\in\hat{\Lambda}_k'^{,\mathrm{ext}}\\ \hat{x}_1\in k[\sigma^i+\eta,\sigma^{i+1}-\eta]}}kW_{\rm tot}^{(k)}\left(\hat{x}',\vec{y}^{\,(k)}(\hat{x})\right)\geq \frac{1}{2} \int_{\sigma^i+\eta}^{\sigma^{i+1}-\eta} Q_3^{\rm rel}(\T{R}\pl_{\xb_1}R)\md\xb_1,\quad i=0,1,\dots,\bar{n}_{\rm f},
\end{equation*}
holds true. Since this is fulfilled for any $\eta$, we can let $\eta\goto 0+$ and use the monotone convergence theorem, as we will see later. 

\subsection{Second step --  \texorpdfstring{$\bm{w^*}$}{$w^*$}-limit in measures}
For the crack contribution to the strain energy, we use the \textit{blow-up method} of Fonseca and Müller \cite{blowup}. We will not make a notational distinction between $(\ybk)$ and its hitherto constructed subsequence $(\yb^{(k_j)})$ any more, as this is not relevant for our $\Gamma$-convergence proof.

Now note that $S_{\yb}\cup S_R\subset X_1$, where $X_1=\{\sigma^i\}_{i=1}^{\bar{n}_{\rm f}}$ is from the proof of Theorem \ref{comp}. Write $\tilde{\mathcal{H}}:=\mathcal{H}^0\mres S_{\yb}\cup S_R$. Decomposing $\mu$ into an absolutely continuous part and a singular part, we have
$$\mu=\frac{\md \mu}{\md \tilde{\mathcal{H}}}\tilde{\mathcal{H}}+\mu_{\rm s}$$
with $\mu_{\rm s}\geq 0$. The $w^*$-convergence then gives (cf. \cite[Th.~1.40]{Gariepy})
\begin{equation*}
\liminf_{k\goto\infty}\sum_{i=1}^{\bar{n}_{\rm f}}\sum_{\substack{\hat{x}\in\hat{\Lambda}_k^{',\mathrm{ext}}\\ \hat{x}_1\in k(\sigma^i-\eta,\sigma^i+\eta)}}kW_{\rm tot}^{(k)}\bigl(\hat{x}',\vec{y}^{\,(k)}(\hat{x})\bigr)
\geq \mu\; \Bigl(\bigcup_{i=1}^{\bar{n}_{\rm f}} (\sigma^i-\eta,\sigma^i+\eta) \Bigr)
\geq \sum_{\sigma\in S_{\yb}\cup S_R}\frac{\md \mu}{\md \tilde{\mathcal{H}}}(\sigma).
\end{equation*}
The goal now is to find the \textit{asymptotic minimal energy $\varphi=\varphi(\yb^+-\yb^-,(R^-)^{-1}R^+)$ necessary to produce a crack or kink} and for every $1\leq i\leq n_{\rm f}$, show that
\begin{equation*}
\frac{\md \mu}{\md \tilde{\mathcal{H}}}(\sigma^i)\geq \varphi\bigl(\yb(\sigma^i+)-\yb(\sigma^i-),(R(\sigma^i-))^{-1}R(\sigma^i+)\bigr).
\end{equation*}
Let us expand the definition of the derivative of $\mu$:
\begin{equation*}
\frac{\md \mu}{\md \tilde{\mathcal{H}}}(\sigma^i)\eqdef\lim_{r\goto 0+}\frac{\mu([\sigma^i-r,\sigma^i+r])}{\tilde{\mathcal{H}}([\sigma^i-r,\sigma^i+r])}=\lim_{r\goto 0+}\frac{\mu([\sigma^i-r,\sigma^i+r])}{1}.
\end{equation*}
By \cite[Prop.~1.15]{Fonseca} and \cite[Th.~1.40]{Gariepy}, we can find $r_n\searrow 0$
such that
\begin{align*}
\frac{\md \mu}{\md \tilde{\mathcal{H}}}(\sigma^i)
&=\lim_{n\goto\infty}\lim_{k\goto\infty}\mu_k((\sigma^i-r_n,\sigma^i+r_n))\\
&=\lim_{n\goto\infty}\lim_{k\goto\infty}\sum_{\substack{\hat{x}\in\hat{\Lambda}_k'^{,\mathrm{ext}}\\ \hat{x}_1\in k(\sigma^i-r_n,\sigma^i+r_n)}}kW_{\rm tot}^{(k)}\bigl(\hat{x}',\vec{y}^{\,(k)}(\hat{x})\bigr).
\end{align*}

\subsection{Third step -- preliminary cell formula obtained by blowup}

First we shall find a preliminary lower bound $\psi$ by rescaling $(\sigma^i-r_n,\sigma^i+r_n)$ to a fixed interval (cf. \cite[proof of Theorem~5.14, Step 3]{SBVbook}). There is a sequence $(k_n)_{n=1}^\infty$ such that $k_n\geq n$, $r_nk_n\goto\infty$,
\begin{align*}
\frac{\md \mu}{\md \tilde{\mathcal{H}}}(\sigma^i)=\lim_{n\goto\infty}\sum_{\substack{\hat{x}\in\hat{\Lambda}_{k_n}'^{,\mathrm{ext}}\\ \hat{x}_1\in k_n( \sigma^i-r_n,\sigma^i+r_n)}}k_nW_{\rm tot}^{(k_n)}\bigl(\hat{x}',\vec{y}^{\,(k_n)}(\hat{x})\bigr), 
\end{align*} 
as well as 
\begin{align} 
\int_{(\sigma^i-2r_n,\sigma^i+2r_n)\times S^{\rm ext}}|\yb^{(k_n)}-\yb|^2\md\xb_1\md\xb'+
\int_{\{\sigma^i+[ (-2r_n,-\frac{1}{4}r_n)\cup(\frac{1}{4}r_n,2r_n)]\}\times S^{\rm ext}}|\nabla_{k_n}\yb^{(k_n)}-R|^2\md\xb\leq r_n^2
\label{eq:diagConv} 
\end{align} 
and $\sigma^i-\frac{r_n}{2}+\frac{2}{k_n}<s_{k_n}^j<\sigma^i+\frac{r_n}{2}-\frac{2}{k_n}$ for every $n\in\N$ and each of the (finitely many) sequences $(s_{k_n}^j)_{n=1}^\infty$ of midpoints of broken slices satisfying $\lim_{n\goto\infty} s_{k_n}^j=\sigma^i$. Since the restrictions of $\yb$ and $R$ to left and right neighbourhoods of $\sigma^i$ are $H^1$, we get for the rescaled functions 
\begin{align*}
y^{\ddagger,n}(w_1)&:=\yb(\sigma^i+r_nw_1),\\
R^{\ddagger,n}(w_1)&:=R(\sigma^i+r_nw_1),\;w_1\in[-1,1],
\end{align*}
the convergences $y^{\ddagger,n}\goto y_{\rm PC}$ in $L^2([-1,1];\R^3)$ and $R^{\ddagger,n}\goto R_{\rm PC}$ in $L^2([-1,1];\R^{3\times 3})$ for $n\goto\infty$, where the piecewise constant functions $y_{\rm PC}$, $R_{\rm PC}$ are defined through
\begin{align*}
y_{\rm PC}(w_1):=\begin{cases}
\yb(\sigma^i-)=\yb^- & w_1<0,\\
\yb(\sigma^i+)=\yb^+ & w_1\geq 0,
\end{cases} 
\quad\text{ and }\quad 
R_{\rm PC}(w_1):=\begin{cases}
R(\sigma^i-)=R^- & w_1<0,\\
R(\sigma^i+)=R^+ & w_1\geq 0.
\end{cases}
\end{align*}
We also set, for $w_1\in[-1,1]$,
\begin{align*}
\mathscr{y}^{(k_n)}(w_1,\xb')
&:=\yb^{(k_n)}(\sigma^i_{k_n}+r_nw_1,\xb'),\\
\nabla_{r_n,k_n}\mathscr{y}^{(k_n)}(w_1,\xb')
&:=\bigl(\frac{1}{r_n}\pl_{w_1}\mathscr{y}^{(k_n)}|k_n\pl_{\xb_2}\mathscr{y}^{(k_n)}|k_n\pl_{\xb_3}\mathscr{y}^{(k_n)}\bigr)
=\nabla_{k_n}\yb^{(k_n)}(\sigma^i_{k_n}+r_nw_1,\xb'),
\end{align*}
where $\sigma^i_{k_n}=\frac{1}{k_n}\lfloor {k_n} \sigma^i \rfloor$. 
Then using \eqref{eq:diagConv}, we get $\mathscr{y}^{(k_n)}\goto y_{\rm PC}$ in $L^2([-1,1]\times S^{\rm ext};\R^3)$ and $\nabla_{r_n,k_n}\mathscr{y}^{(k_n)}\goto R_{\rm PC}$ in $L^2([I_\psi^-\cup I_\psi^+]\times S^{\rm ext}];\R^{3\times 3})$, where $I_\psi^-=[-1,-\frac{1}{2}]$ and $I_\psi^+=[\frac{1}{2},1]$. This gives the preliminary estimate with `converging boundary conditions':
\begin{align*}
\frac{\md \mu}{\md \tilde{\mathcal{H}}}(\sigma^i)
&\geq\min\Bigl\{\limsup_{n\goto\infty}\sum_{(w_1,\xb')\in\Lambda_{r_n,k_n}'}k_nW_{\rm tot}^{(k_n)}\bigl(\xb',\vec{\mathscr{y}}^{\,(k_n)}(w_1,\xb')\bigr); \\ 
&\qquad \qquad \mathscr{y}^{(k_n)}\in\mathrm{PAff}(\Lambda_{r_n,k_n}),\; r_n\searrow 0,\; r_nk_n\goto\infty,\; \\ 
&\qquad \qquad ||\mathscr{y}^{(k_n)}-y_{\rm PC}||_{L^2(I_\psi^\pm\times S^{\rm ext})}\goto 0,\;||\nabla_{r_n,k_n}\mathscr{y}^{(k_n)}-R_{\rm PC}||_{L^2(I_\psi^\pm\times S^{\rm ext})}\goto 0 \Bigr\} \\
&=:\tilde{\psi}(\yb^-,\yb^+,R^-,R^+), 
\end{align*}
where 
\begin{align*}
\vec{\mathscr{y}}^{\,(k_n)}(w_1,\xb')
&:=k_n\bigl(\mathscr{y}^{(k_n)}\bigl(w_1+\frac{1}{r_nk_n}\zf^i_1,\xb'+(\zf^i)'\bigr)\bigr)_{i=1}^8,\\
\Lambda_{r_n,k_n}
&:=\Bigl(\frac{1}{r_nk_n}\Z \cap \bigl(-1-\frac{1}{r_nk_n}, 1+\frac{1}{r_nk_n}\bigr)\Bigr)\times\mathcal{L}^{\rm ext},\\
\Lambda_{r_n,k_n}'
&:=\Bigl(\bigl(\frac{1}{2r_nk_n}+\frac{1}{r_nk_n}\Z \bigr) \cap \bigl(-1-\frac{1}{2r_nk_n}, 1+\frac{1}{2r_nk_n}\bigr)\Bigr)\times\mathcal{L}'^{,\rm ext},
\end{align*}
and $\mathrm{PAff}(\Lambda_{r_n,k_n})$ denotes the class of piecewise affine mappings $\mathscr{v}\colon [-1-\frac{1}{r_nk_n}, 1+\frac{1}{r_nk_n}] \times \overline{S^{\rm ext}}\goto\R^3$ which are generated by interpolating their values from $\Lambda_{r_n,k_n}$ by the scheme from Subsection \ref{sec:rie}. The minimum in $\tilde{\psi}$ runs over all sequences $\{r_n\}\subset (0,\infty)$, $\{k_n\}\subset\N$ and $(\mathscr{y}^{(k_n)})$ with the above properties.

It can be shown by a diagonalization argument that the minimum is attained; this is also the case in \eqref{eq:phi}. From the translation and rotation invariance of $W_{\rm cell}^{(k)}$ we see that $\tilde{\psi}(\yb^-,\yb^+,R^-,R^+)=\psi(\yb^+-\yb^-,(R^-)^{-1}R^+)$ for a function $\psi\colon \R^3\times{\rm SO}(3)\goto [0,\infty]$.

\subsection{Fourth step -- rigid boundary conditions in the cell formula}\label{sec:finCell}

At last, we relate the preliminary cell formula $\psi$ to the final cell formula which uses rigid boundary conditions instead of $L^2$-converging ones:
\begin{align}\label{eq:phi}
\begin{split}
\ph\bigl(\yb^+-\yb^-,(R^-)^{-1}R^+\bigr)
&=\min\Bigl\{\limsup_{n\goto\infty}\sum_{(w_1,\xb')\in\Lambda_{r_n,k_n}'}k_nW_{\rm tot}^{(k_n)}\bigl(\xb',\vec{\mathscr{y}}^{\,(k_n)}(w_1,\xb')\bigr);\\
&\qquad\qquad \bigl((r_n)_{n=1}^\infty,(k_n)_{n=1}^\infty,(\mathscr{y}^{(k_n)})_{n=1}^\infty\bigr)\in\mathcal{V}_{\yb^+-\yb^-,(R^-)^{-1}R^+}\Bigr\}
\end{split}
\end{align}
with 
\begin{align*}\label{eq:V}
\begin{split}
\mathcal{V}_{\yb^+-\yb^-,(R^-)^{-1}R^+}
&=\Bigl\{\bigl((r_n)_{n=1}^\infty,(k_n)_{n=1}^\infty,(\mathscr{y}^{(k_n)})_{n=1}^\infty\bigr)\in (0,\infty)^\N\times\N^\N\times\mathrm{PAff}(\Lambda_{r_n,k_n})^\N;\\
&\qquad\mathscr{y}^{(k_n)}(w_1,\xb')=R_\pm^{(k_n)}\T{\bigl(r_nw_1,\frac{1}{k_n}\xb'\bigr)}+y_\pm^{(k_n)}\text{ on }I^\pm\times \overline{S^{\rm ext}},\ r_n\searrow 0,\\
&\qquad\qquad r_nk_n\goto\infty,\; y_\pm^{(k_n)}\in\R^3,\; R_\pm^{(k_n)}\in{\rm SO}(3),\; y_\pm^{(k_n)}\goto \yb^\pm,\; R_\pm^{(k_n)}\goto R^\pm \Bigr\},
\end{split}
\end{align*}
$I^-=[-1,-\frac{3}{4}]$ and $I^+=[\frac{3}{4},1]$. 
\begin{rem}
The particular choice 
$$ \mathscr{y}^{(k_n)}(w_1,\xb')
  =\begin{cases}
    R_-^{(k_n)}\T{(r_nw_1,\frac{1}{k_n}\xb')}+y_-^{(k_n)}
    &\text{if } w_1 \le 0, \\ 
    R_+^{(k_n)}\T{(r_nw_1,\frac{1}{k_n}\xb')}+y_+^{(k_n)}
    &\text{if } w_1 > 0
    \end{cases} $$ 
for given $\yb^+,\yb^-\in\R^3$ and $R^-,R^+\in\mathrm{SO}(3)$ shows that, in case \tc{(W8)} holds true, one has $\varphi\le \bar{C}_1$ for some $\bar{C}_1< \infty. $
\end{rem}
We now show that we have $\psi\geq\ph$. Suppose $\e>0$ and that $(\mathscr{y}^{(k_n)})_{n=1}^\infty$ is a sequence $\mathrm{PAff}(\Lambda_{r_n,k_n})$ such that 
\begin{equation}\label{eq:Paff-condn}
||\mathscr{y}^{(k_n)}-y_{\rm PC}||_{L^2(I_\psi^\pm\times S^{\rm ext})}\goto 0,\quad 
||\nabla_{r_n,k_n}\mathscr{y}^{(k_n)}-R_{\rm PC}||_{L^2(I_\psi^\pm\times S^{\rm ext})}\goto 0
\end{equation}
and 
\begin{equation*}
\limsup_{n\goto\infty}\calE_{k_n}(\mathscr{y}^{(k_n)},[-1,1])\leq \psi(\yb^+-\yb^-,(R^-)^{-1}R^+)+\e,
\end{equation*}
where for any $I\subset[-1,1]$ we set
\begin{equation*}
\calE_{k_n}(\mathscr{y}^{(k_n)},I):=\sum_{\substack{w_1 \in  \mathscr{L}_n'(I)\\ \xb' \in \mathcal{L}'^{,\rm ext}}}k_nW_{\rm tot}^{(k_n)}\bigl(\xb',\vec{\mathscr{y}}^{\,(k_n)}(w_1,\xb')\bigr)
\end{equation*}
and $\mathscr{L}_n'(I) = (\frac{1}{2r_nk_n}+\frac{1}{r_nk_n} \Z) \cap I$. The definition of a rod slice in this section reads
$$S_{k_n}(w_1)=\bigl[\bar{w}_1-\frac{1}{2r_nk_n},\bar{w}_1+\frac{1}{2r_nk_n}\bigr)\times\overline{S^{\rm ext}},
\quad\text{where}\quad 
\bar{w}_1 = \frac{1}{r_nk_n} \lfloor r_nk_n w_1 \rfloor + \frac{1}{2r_nk_n}.$$
Our goal now is to find a sequence $\mathscr{v}^{(k_n)}$ which is admissible as a competitor in the definition of $\ph$ and has asymptotically lower energy than $\mathscr{y}^{(k_n)}$. We provide the construction only for $\mathscr{v}^{(k_n)}|_{[-1,0]\times \overline{S^{\rm ext}}}$, as for $\mathscr{v}^{(k_n)}|_{(0,1]\times \overline{S^{\rm ext}}}$ we could proceed analogously. Writing $I_{0,n}^-:=\frac{1}{r_nk_n}(\lfloor -\frac{3}{4}r_nk_n\rfloor+1,\lfloor -\frac{1}{2}r_nk_n\rfloor)$ for a discrete approximation of $I_\psi^-\setminus I^-$ from inside and $N_n^-=\lfloor -\frac{1}{2}r_nk_n\rfloor-\lfloor -\frac{3}{4}r_nk_n\rfloor-3 = \sharp\mathscr{L}'(I_{0,n}^-)-2$ for the number of (interior) slices intersecting $I_{0,n}^-\times \overline{S^{\rm ext}}$, we introduce the sets
\begin{subequations}\label{eq:deGiSets}
\begin{align}
\begin{split}
W_1^{(n)}
&=\Bigl\{w_1\in\mathscr{L}'(I_{0,n}^-);\; w_1\pm\tfrac{i}{r_nk_n} \in\mathscr{L}'(I_{0,n}^-), \\
&\qquad \sum_{i\in\{-1,0,1\}} \sum_{\xb'\in\mathcal{L}'^{,\rm ext}}k_nW_{\rm tot}^{(k_n)}\bigl(\xb',\bar{\nabla}_{r_n,k_n}\mathscr{y}^{(k_n)}(w_1+\tfrac{i}{r_nk_n},\xb')\bigr)\leq\frac{12}{N_n^-}\calE_{k_n}(\mathscr{y}^{(k_n)},I_{0,n}^-)\Bigr\},
\end{split}\label{eq:deGiE}\\
\begin{split}
W_2^{(n)}
&=\Bigl\{w_1\in\mathscr{L}'(I_{0,n}^-);\; w_1\pm\tfrac{i}{r_nk_n} \in\mathscr{L}'(I_{0,n}^-), \\
&\qquad \int_{S_{k_n}(w_1)}|\nabla_{r_n,k_n}\mathscr{y}^{(k_n)}-R^-|^2\md w_1\md\xb'\leq\frac{4}{N_n^-}||\nabla_{r_n,k_n}\mathscr{y}^{(k_n)}-R^-||_{L^2(I_{0,n}^-\times S^{\rm ext};\R^{3\times 3})}^2\Bigr\},\end{split}\label{eq:deGiR}\\
\begin{split}
W_3^{(n)}
&=\Bigl\{w_1\in\mathscr{L}'(I_{0,n}^-);\; w_1\pm\tfrac{i}{r_nk_n} \in\mathscr{L}'(I_{0,n}^-), \\
&\qquad 
\int_{S_{k_n}(w_1)}|\mathscr{y}^{(k_n)}-\yb^-|^2\md w_1\md\xb'\leq\frac{4}{N_n^-}||\mathscr{y}^{(k_n)}-\yb^-||_{L^2(I_{0,n}^-\times S^{\rm ext};\R^3)}^2\Bigr\},\end{split}\label{eq:deGiy}
\end{align}
\end{subequations}
where $\bar{\nabla}_{r_n,k_n}\mathscr{y}^{(k_n)}(w_1,\xb')=k_n(\mathscr{y}^{(k_n)}(\bar{w}_1+\frac{1}{r_nk_n}\zf_1^i,\bar{\xb}'+(\zf^i)')-\sum_{j=1}^8\mathscr{y}^{(k_n)}(\bar{w}_1+\frac{1}{r_nk_n}\zf_1^j,\bar{\xb}'+(\zf^j)'))_{i=1}^8$. The sets $W_i^{(n)}$, $i=1,2,3$, are comprised of the midpoints of the $w_1$-projections of slices on which, \tc{loosely speaking,} a certain quantity is below four times its average.  By Lemma \ref{DeGiorgiTrick} with $p=4$ we see that for every $i\in\{1,2,3\}$ and $n\in\N$, the set $W_i^{(n)}$ contains at least $\lfloor(3/4)N_n^-\rfloor$ elements. The pigeonhole principle then implies that for every $n$ large enough there is $w_-^{(n)}\in W_1^{(n)}\cap W_2^{(n)}\cap W_3^{(n)}$. Since $N_n^-\geq \frac{1}{4}r_nk_n-4$, the inequality in \eqref{eq:deGiE} and the finiteness in \eqref{eq:liminfEbdd} imply an estimate in integral form:
\begin{equation}\label{eq:deGiEint}
\tc{\sum_{i\in\{-1,0,1\}}} r_nk_n\int_{S_{k_n}({\scriptstyle w_-^{(n)}+\frac{i}{r_nk_n}})}k_nW_{\rm tot}^{(k_n)}\bigl(\xb',\bar{\nabla}_{r_n,k_n}\mathscr{y}^{(k_n)}\bigr)\md w_1\md\xb'
\leq \tc{\frac{48}{r_n k_n-16}}\calE_{k_n}(\mathscr{y}^{(k_n)},I_{0,n}^-)
\leq \frac{C_{\rm e}}{r_nk_n}
\end{equation}
for a constant $C_{\rm e}>0$. Hence we can employ the growth assumption on the elastic cell energy $W_0$\tc{, properties of the extension scheme (cf. \eqref{eq:nbhd-rig-est}),} and \cite[Theorem~3.1]{FrM02} (in unrescaled variables) to get $R_-^{(k_n)}\in{\rm SO}(3)$ such that
\begin{equation*}
\frac{1}{C}||\nabla_{r_n,k_n}\mathscr{y}^{(k_n)}-R_-^{(k_n)}||_{L^2(S_{k_n}({\scriptstyle w_-^{(n)}});\R^{3\times 3})}^2\leq \sum_{i\in\{-1,0,1\}} \int_{S_{k_n}({\scriptstyle w_-^{(n)}+\frac{i}{r_nk_n}})}\tc{W_{\rm tot}^{(k_n)}\bigl(\xb'},\bar{\nabla}_{r_n,k_n}\mathscr{y}^{(k_n)}\bigr)\md w_1\md\xb'
\end{equation*}
for a constant $C>0$. Combining the previous inequality with \eqref{eq:deGiEint} we deduce that
\begin{equation}\label{eq:yRrate}
||\nabla_{r_n,k_n}\mathscr{y}^{(k_n)}-R_-^{(k_n)}||_{L^2(S_{k_n}({\scriptstyle w_-^{(n)}});\R^{3\times 3})}=O\Bigl(\frac{1}{r_nk_n^{3/2}}\Bigr).
\end{equation}
Setting
\begin{equation*}
y_-^{(k_n)}=\dashint_{S_{k_n}({\scriptstyle w_-^{(n)}})}\mathscr{y}^{(k_n)}(w_1,\xb')-R_-^{(k_n)}\T{\bigl(r_nw_1,\frac{1}{k_n}\xb'\bigr)}\md w_1\md\xb',
\end{equation*}
we achieve that a Poincaré inequality is satisfied, with a $C>0$:
\begin{align}\label{eq:yRPoinc}
\begin{split}
\MoveEqLeft 
\sqrt{\int_{S_{k_n}({\scriptstyle w_-^{(n)}})}|\mathscr{y}^{(k_n)}(w_1,\xb') - {\displaystyle R_-^{(k_n)}} \T{\bigl(r_nw_1,\frac{1}{k_n}\xb'\bigr)}-{\displaystyle y_-^{(k_n)}}|^2\md w_1\md\xb'
}\\
&\leq C\frac{1}{k_n}||\nabla_{r_n,k_n}\mathscr{y}^{(k_n)}-R_-^{(k_n)}||_{L^2(S_{k_n}({\scriptstyle w_-^{(n)}});\R^{3\times 3})}.
\end{split}
\end{align}
Define $\mathscr{v}^{(k_n)}\colon [-1,0]\times \overline{S^{\rm ext}}\goto \R^3$ as follows:
\begin{equation*}
\mathscr{v}^{(k_n)}(w_1,\xb')
=\begin{cases}
R_-^{(k_n)}\T{(r_nw_1, \frac{1}{k_n}\xb')}+y_-^{(k_n)}&-1 \leq w_1\leq w_-^{(n)}-\frac{1}{2r_nk_n}\\[1ex]
\text{pcw. affine (24 simplices/cell)}& 
w_-^{(n)}-\frac{1}{2r_nk_n}<w_1<w_-^{(n)}+\frac{1}{2r_nk_n}\\
\mathscr{y}^{(k_n)}(w_1,\xb')& 0 \geq w_1\geq w_-^{(n)}+\frac{1}{2r_nk_n}.
\end{cases}
\end{equation*}
We claim that
\begin{align}
\limsup_{n\goto\infty}\calE_{k_n}\bigl(\mathscr{v}^{(k_n)},[-1,0)\bigr)&\leq \limsup_{n\goto\infty}\calE_{k_n}\bigl(\mathscr{y}^{(k_n)},[-1,0)\bigr),\label{eq:forFleqPs}\\
\lim_{n\goto\infty} y_-^{(k_n)}=\yb^-,&\quad \lim_{n\goto\infty} R_-^{(k_n)}=R^-.\label{eq:convBC}
\end{align}
Concerning \eqref{eq:forFleqPs}, we notice that for all $n\in\N$,
\begin{equation*}
\calE_{k_n}\Bigl(\mathscr{y}^{(k_n)},\bigl(w_-^{(n)}+\frac{1}{2r_nk_n},0\bigl)\Bigr)=\calE_{k_n}\Bigl(\mathscr{v}^{(k_n)},\bigl(w_-^{(n)}+\frac{1}{2r_nk_n},0\bigr)\Bigr)
\end{equation*}
and that $\calE_{k_n}(\mathscr{v}^{(k_n)},(-1,w_-^{(n)}-\frac{1}{2r_nk_n}))=0$ since $\bar{\nabla}_{r_n,k_n}\mathscr{v}^{(k_n)}=R_-^{(k_n)}\bar{\Id}\in \bar{\rm SO}(3)$ on $(-1,w_-^{(n)}-\frac{1}{2r_nk_n})\times S^{\rm ext}$. Hence it remains to show that the energy on the transition slice $S_{k_n}(w_-^{(n)})$ vanishes in the limit.
\begin{lemma} The following is true:
\begin{equation*}
\lim_{n\goto\infty}\calE_{k_n}\Bigl(\mathscr{y}^{(k_n)},w_-^{(n)}+\frac{1}{2r_nk_n}\bigl(-1,1\bigr)\Bigr)+\calE_{k_n}\Bigl(\mathscr{v}^{(k_n)},w_-^{(n)}+\frac{1}{2r_nk_n}\bigl(-1,1\bigr)\Bigr)=0.
\end{equation*}
\end{lemma}
\begin{proof}The proof is divided into several steps. Let $Q=[w_-^{(n)}-\frac{1}{2r_nk_n},w_-^{(n)}+\frac{1}{2r_nk_n}]\times Q'$, where $Q'=\xb'+[-\frac{1}{2},\frac{1}{2}]^2$ for some \tc{$\xb'\in\mathcal{L}'^{,\mathrm{ext}}$}, be any atomic cell contained in the slice $\overline{S_{k_n}}(w_-^{(n)})$.

\textit{Step} 1. Using \cite[Lemma~3.5]{BS09} and \eqref{eq:yRrate}, we can obtain the relation
\begin{equation}\label{eq:y-Rkn-minus}
c|\bar{\nabla}_{r_n,k_n}\mathscr{y}^{(k_n)}(w_-^{(n)},\xb')-R_-^{(k_n)}\bar{\Id}|^2
\leq r_nk_n\int_{Q}|\nabla_{r_n,k_n}\mathscr{y}^{(k_n)}-R_-^{(k_n)}|^2\md w_1\tc{\md w'}=O\Bigl(\frac{1}{r_nk_n^2}\Bigr)
\end{equation}
with a constant $c>0$.

\textit{Step} 2. We now compare $\vec{\mathscr{y}}^{\,(k_n)}(w_1,\xb')$ and $\vec{\mathscr{v}}^{\,(k_n)}(w_1,\xb')$. By construction we have $[\vec{\mathscr{y}}^{(k_n)}]_{\cdot i}=[\vec{\mathscr{v}}^{(k_n)}]_{\cdot i}$ for $i=5,6,7,8$ and from Step 1 we get, for $i=1,2,3,4$,
\begin{align*}
\begin{split}
\MoveEqLeft{} 
[\vec{\mathscr{y}}^{(k_n)}(w_1,\xb')]_{\cdot i}-[\vec{\mathscr{v}}^{(k_n)}(w_1,\xb')]_{\cdot i} \\
&=\Bigl|k_n\Bigl(\mathscr{y}^{(k_n)}\bigl(w_-^{(n)}+\frac{1}{r_nk_n}\zf_1^i,\xb'+(\zf^i)'\bigr)-R_-^{(k_n)}\T{\bigl(r_nw_-^{(n)}+\frac{1}{k_n}\zf_1^i,\; 
\frac{1}{k_n}(\xb+\zf^i)'\bigr)}-y_-^{(k_n)}\Bigr)\Bigr|\\
&\leq \underbrace{\bigl|[\bar{\nabla}_{r_n,k_n}\mathscr{y}^{(k_n)}(w_-^{(n)},\xb')]_{\cdot i}-R_-^{(k_n)}\zf^i\bigr|}_{=O(r_n^{-1/2}k_n^{-1})}
+k_n\bigl|\langle\mathscr{y}^{(k_n)}\rangle-R_-^{(k_n)}\T{\bigl(r_nw_-^{(n)},\frac{1}{k_n}\xb'\bigr)}-y_-^{(k_n)}\bigr|.
\end{split}
\end{align*}
Property \eqref{eq:surfVolMean} of our piecewise affine interpolation, Hölder's inequality, \eqref{eq:yRPoinc} and \eqref{eq:yRrate} give
\begin{align*}
\MoveEqLeft
k_n\bigl|\langle\mathscr{y}^{(k_n)}\rangle-R_-^{(k_n)}\T{\bigl(r_nw_-^{(n)},\frac{1}{k_n}\xb'\bigr)}-y_-^{(k_n)}\bigr|\\
&=r_nk_n^2\Bigl|\int\limits_{Q}\mathscr{y}^{(k_n)}(w)-R_-^{(k_n)}\T{\bigl(r_nw_1,\frac{1}{k_n} w'\bigr)}-y_-^{(k_n)}\md w_1\md w'\Bigr|\\
&\leq C\sqrt{|Q|}r_nk_n^2\frac{1}{k_n}||\nabla_{r_n,k_n}\mathscr{y}^{(k_n)}-R_-^{(k_n)}||_{L^2(S_{k_n}({\scriptstyle w_-^{(n)}});\R^{3\times 3})}=O\Bigl(\frac{1}{\sqrt{r_n}k_n}\Bigr)
\end{align*}
so that $|\vec{\mathscr{y}}^{\,(k_n)}(w_1,\xb')-\vec{\mathscr{v}}^{\,(k_n)}(w_1,\xb')|=O(r_n^{-1/2}k_n^{-1})$ and, in particular, 
\begin{align*}
|\bar{\nabla}_{r_n,k_n}\mathscr{y}^{(k_n)}(w_-^{(n)},\xb') - \bar{\nabla}_{r_n,k_n}\mathscr{v}^{(k_n)}(w_-^{(n)},\xb')|
=O\Bigl(\frac{1}{\sqrt{r_n}k_n}\Bigr)
\end{align*}
since $\bar{\nabla}_{r_n,k_n}\mathscr{y}^{(k_n)}(w_-^{(n)},\xb')=\vec{\mathscr{y}}^{\,(k_n)}(w_1,\xb')-\frac{1}{8}\sum_{i=1}^8 [\vec{\mathscr{y}}^{\,(k_n)}(w_1,\xb')]_{\cdot i}(1, \ldots, 1)$ and likewise for $\mathscr{v}^{(k_n)}$. Together with \eqref{eq:y-Rkn-minus} this shows that also $\mathscr{v}^{(k_n)}$ satisfies 
\begin{equation}\label{eq:v-Rkn-minus}
\tc{|\bar{\nabla}_{r_n,k_n}\mathscr{v}^{(k_n)}(w_-^{(n)},\xb')-R_-^{(k_n)}\bar{\Id}|}
=O\Bigl(\frac{1}{\sqrt{r_n}k_n}\Bigr).
\end{equation}

\textit{Step} 3. Now we use that $W_{\rm tot}^{(k_n)}$ is independent of $k_n$ on a tubular neighbourhood of $\mathrm{SO}(3)$ of size $O(k_n^{-1})$ and, by Taylor expansion, satisfies an estimate of the form $W_{\rm tot}^{(k_n)} \le C \mathrm{dist}^2(\cdot, \mathrm{SO}(3))$ there. Thus, \eqref{eq:y-Rkn-minus} and \eqref{eq:v-Rkn-minus} give 
\begin{align*}
 k_nW_{\rm tot}^{(k_n)}\bigl(\xb',\bar{\nabla}_{r_n,k_n}\mathscr{y}^{(k_n)}\bigr)
+ k_nW_{\rm tot}^{(k_n)}\bigl(\xb',\bar{\nabla}_{r_n,k_n}\mathscr{v}^{(k_n)}\bigr) 
= O\Bigl(\frac{1}{r_nk_n}\Bigr).
\end{align*}
This implies the assertion.
\end{proof}

The second convergence in \eqref{eq:convBC} is a consequence of \eqref{eq:deGiR}, \eqref{eq:Paff-condn}, and \eqref{eq:yRrate}:
\begin{align*}
|R_-^{(k_n)}-R^-|^2
&=\frac{r_nk_n}{|S^{\rm ext}|}\int_{S_{k_n}({\scriptstyle w_-^{(n)}})}|R_-^{(k_n)}-R^-|^2\md w_1\md\xb' \\
&\leq\frac{2r_nk_n}{|S^{\rm ext}|}\Bigl(\int_{S_{k_n}({\scriptstyle w_-^{(n)}})}|R^--\nabla_{r_n,k_n}\mathscr{y}^{(k_n)}|^2\md w_1\md\xb'+\int_{S_{k_n}({\scriptstyle w_-^{(n)}})}|R_-^{(k_n)}-\nabla_{r_n,k_n}\mathscr{y}^{(k_n)}|^2\md w_1\md\xb'\Bigr) \\
&\leq\frac{2r_nk_n}{|S^{\rm ext}|}\cdot\frac{4}{\tc{\frac{1}{4}r_nk_n-4}}||\nabla_{r_n,k_n}\mathscr{y}^{(k_n)}-R^-||_{L^2(I_{0,n}^-\times S;\R^{3\times 3})}^2+O\Bigl(\frac{1}{r_nk_n^2}\Bigr)\longrightarrow 0.
\end{align*}
The first convergence in \eqref{eq:convBC} follows similarly from \eqref{eq:deGiy} \tc{and \eqref{eq:Paff-condn}} if we use \eqref{eq:yRPoinc} and \eqref{eq:yRrate} to show that
\begin{align*}
\MoveEqLeft\frac{2r_nk_n}{|S^{\rm ext}|}\int\limits_{S_{k_n}({\scriptstyle w_-^{(n)}})}|y_-^{(k_n)}-\mathscr{y}^{(k_n)}|^2\md w_1\md\xb' \\ 
&\leq C\bigl[r_n||\nabla_{r_n,k_n}\mathscr{y}^{(k_n)}-R_-^{(k_n)}||_{L^2(S_{k_n}({\scriptstyle w_-^{(n)}});\R^{3\times 3})}^2+|R_-^{(k_n)}|^2r_nk_n\frac{1}{|\tc{S^{\rm ext}}|r_nk_n}\bigl|\bigl(r_n,\frac{1}{k_n},\frac{1}{k_n}\bigr)\bigr|^2\bigr]\longrightarrow 0,
\end{align*}
with a constant $C>0$.

In the same way, we could construct $(R_+^{(k_n)})_{n=1}^\infty$, $(y_+^{(k_n)})_{n=1}^\infty$, and $\mathscr{v}^{(k_n)}|_{(0,1]\times S^{\rm ext}}$ and prove a version of \eqref{eq:forFleqPs}--\eqref{eq:convBC} on $(0,1]$. Thus, as
\begin{align*}
\ph\bigl(\yb^+-\yb^-,(R^-)^{-1}R^+\bigr)
&\leq \limsup_{n\goto\infty} \calE_{k_n}\bigl(\mathscr{v}^{(k_n)},[-1,1]\bigr)\\
&\leq \limsup_{n\goto\infty} \calE_{k_n}\bigl(\mathscr{y}^{(k_n)},[-1,1]\bigr) 
\leq\psi\bigl(\yb^+-\yb^-,(R^-)^{-1}R^+\bigr)+\e
\end{align*}
and $\e>0$ was arbitrary, the claim that $\ph\leq\psi$ is proved.
\begin{lemma}\label{DeGiorgiTrick}
Let $c_1$, $c_2$, $\dots$, $c_N$ be nonnegative reals and $p\geq 1$. Then
$$\sharp\Bigl\{i\in\{1,\dots,N\};\;c_i\leq\frac{p}{N}\sum_{j=1}^N c_j\Bigr\}>\bigl\lfloor\bigl(1-\frac{1}{p}\bigr)N\bigr\rfloor.$$
\end{lemma}
\begin{proof}
We denote by $\bar{c}$ the average $N^{-1}\sum_j c_j$. If the statement were not true, the number of $c_j$'s such that $c_j>p\bar{c}$ would be greater than or equal to $N/p$. Hence
$$\bar{c} \ge \frac{1}{N}\sum_{j;\; c_j>p\bar{c}}c_j>\frac{1}{N}p\bar{c}\frac{N}{p}=\bar{c},$$
but that is a contradiction.
\end{proof}

Summing up the elastic and crack energy contributions, we get
\begin{multline*}
\lim_{k\goto\infty}kE^{(k)}(y^{(k)})\geq\liminf_{k\goto\infty}k\Bigl[\sum_{i=0}^{\bar{n}_{\rm f}}\sum_{\substack{\hat{x}\in\hat{\Lambda}_k'^{,\mathrm{ext}}\\\hat{x}_1\in k[\sigma^i+\eta,\sigma^{i+1}-\eta]}}W_{\rm tot}^{(k)}\left(\hat{x}',\vec{y}^{\,(k)}(\hat{x})\right)\\
+\sum_{i=1}^{\bar{n}_{\rm f}}\sum_{\substack{\hat{x}\in\hat{\Lambda}_k'^{,\mathrm{ext}}\\\hat{x}_1\in k(\sigma^i-\eta,\sigma^i+\eta)}}W_{\rm tot}^{(k)}\left(\hat{x}',\vec{y}^{\,(k)}(\hat{x})\right)\Bigr]\\
\geq\sum_{i=0}^{\bar{n}_{\rm f}}\frac{1}{2}\int\limits_{\sigma^i+\eta}^{\sigma^{i+1}-\eta} Q_3^{\rm rel}(\T{R}\pl_{\xb_1}R)\md\xb_1+\sum_{\sigma\in S_{\yb}\cup S_R}\varphi\bigl(\yb(\sigma+)-\yb(\sigma-),(R(\sigma-))^{-1}R(\sigma+)\bigr).
\end{multline*}
To obtain the $\Gamma$-liminf inequality, we apply the monotone convergence theorem with $\eta\goto 0+$.

\section{Proof of the upper bound}
For a construction of recovery sequences it is crucial to first analyze the cell formula more precisely. In particular, we will need to prove that the crack set is essentially localized on the atomic scale. 

\subsection{Analysis of the cell formula}

\begin{lemma}[localization of crack]\label{rigCrack}
Let $\yb^-,\yb^+\in\R^3$ and $R^-,R^+\in{\rm SO}(3)$. Then for any $\e_*>0$, there is an $N_*\in\N$, sequences $\{k_n\}_{n=1}^\infty\subset\N$, $\{r_n\}\subset (0,\infty)$ and mappings $\yp{k_n}\in\mathrm{PAff}(\Lambda_{r_n,k_n})$, $n\in\N$, with the following properties:
\begin{align}\label{eq:yPlusE}
\limsup_{n\goto\infty}\calE_{k_n}(\yp{k_n},[-1,1])&\leq \ph\bigl(\yb^+-\yb^-,(R^-)^{-1}R^+\bigr)+\e_*,
\end{align}
$r_n\searrow 0$, $r_nk_n\goto\infty$, and, for suitable $\ypp{\pm}{k_n}\in\R^3,\; \Rp{\pm}{k_n}\in{\rm SO}(3)$ with $\ypp{\pm}{k_n}\goto \yb^\pm$, $\Rp{\pm}{k_n}\goto R^\pm$,
\begin{equation*}
\yp{k_n}(w_1,\xb')=\begin{cases}
\Rp{-}{k_n}\T{\bigl(r_nw_1,\frac{\xb'}{k_n}\bigr)}+\ypp{-}{k_n} & \text{ on }\bigl([-1,0]\setminus I_{\rm c}^{(n)}\bigr)\times \overline{S^{\rm ext}},\\
\Rp{+}{k_n}\T{\bigl(r_nw_1,\frac{\xb'}{k_n}\bigr)}+\ypp{+}{k_n} & \text{ on }\bigl((0,1]\setminus I_{\rm c}^{(n)}\bigr)\times \overline{S^{\rm ext}},
\end{cases}
\end{equation*}
where $I_{\rm c}^{(n)}=\frac{1}{r_nk_n}[-N_*,N_*]$.
\end{lemma}
\begin{proof}
Find $(k_n)_{n=1}^\infty\subset\N$, $(r_n)_{n=1}^\infty\subset (0,\infty)$ with $r_n\searrow 0$ and $\lim_{n\goto\infty}r_n k_n =\infty$, and $(\mathscr{y}^{(k_n)})_{n=1}^\infty\subset\mathrm{PAff}(\Lambda_{r_n,k_n})$ such that
\begin{align*}
\lim_{n\goto\infty}\calE_{k_n}(\mathscr{y}^{(k_n)},[-1,1])= \ph\bigl(\yb^+-\yb^-,(R^-)^{-1}R^+\bigl)
\end{align*}
and, for some $y_\pm^{(k_n)}\in\R^3$, $R_\pm^{(k_n)}\in{\rm SO}(3)$ with $y_\pm^{(k_n)}\goto \yb^\pm$, $R_\pm^{(k_n)}\goto R^\pm$, 
\begin{align*}
\mathscr{y}^{(k_n)}(w_1,\xb')=R_\pm^{(k_n)}\T{\bigl(r_nw_1,\frac{1}{k_n}\xb'\bigr)}+y_\pm^{(k_n)}\text{ on }I^\pm\times \overline{S^{\rm ext}}.
\end{align*}
Recalling assumption \tc{(W5)} on $W_{\rm cell}^{(k_n)}$ and passing to a subsequence (without relabelling it), we can assert that there is an $N_{\rm f}\in\N_0$, $N_{\rm f} \le C\ph(\yb^+-\yb^-,(R^-)^{-1}R^+)$, such that for every $n$, only the slices
\begin{equation*}
S_{k_n}(s_n^j):=\bigl[s_n^j-\frac{1}{2r_nk_n},s_n^j+\frac{1}{2r_nk_n}\bigr)\times \overline{S^{\rm ext}},\quad j\in\{1,\dots,N_{\rm f}\},
\end{equation*}
are \textit{broken} in the sense from the proof of Theorem \ref{comp}, where $s_n^1<\cdots<s_n^{N_{\rm f}}$ are the midpoints of the $w_1$-projections of the broken slices and $\lim_{n\goto\infty}s_n^j=s^j\in[-3/4,3/4]$. This means that $\bar{\nabla}_{r_n,k_n}\mathscr{y}^{(k_n)}$ on the remaining `intact' slices is $c_{\rm frac}^{(k_n)}$-close to $\bar{\rm SO}(3)$. Then
\begin{multline*}
\tilde{I}_1^{(n)}
=\Bigl[-\frac{\lfloor\frac{3}{4}r_nk_n\rfloor}{r_nk_n}\tc{+\frac{1}{r_nk_n}},s_n^1-\frac{1}{2r_nk_n}\Bigr],\\
\tilde{I}_2^{(n)}=\bigl[s_n^1+\frac{1}{2r_nk_n},s_n^2-\frac{1}{2r_nk_n}\bigr],\; \dots,\;\tilde{I}_{N_{\rm f}+1}^{(n)}=\Bigl[s_n^{N_{\rm f}}+\frac{1}{2r_nk_n},\frac{\lfloor\frac{3}{4}r_nk_n\rfloor}{r_nk_n}\Bigr]
\end{multline*}
are the $w_1$-projections of elastically deformed parts of the region surrounding the crack. We fix a number $N_*' \in \N$ (to be determined below) and denote by $\{\tilde{I}_{j_i}^{(n)}\}_{i=1}^{N_{\rm U}}\subset\{\tilde{I}_j^{(n)}\}_{j=1}^{N_{\rm f}+1}$ those intervals $\tilde{I}_{j_i}^{(n)}$ for which $\tg{r_nk_n|\tilde{I}_{j_i}^{(n)}|\geq 2N_*'+4}$. On extracting a further subsequence, $N_{\rm U}=N_{\rm U}(N_*')$ is independent of $n$. We assume $N_{\rm U}>0$, since otherwise the next `rigidification' procedure is redundant and it is enough to construct $\yp{k_n}$ directly from $\mathscr{y}^{(k_n)}$ later. To shorten notation, we set $\tilde{I}_{j_i}^{(n)}=:I_i^{(n)}=[a_i^{(n)}\tg{-\frac{1}{r_nk_n}},b_i^{(n)}\tg{+\frac{1}{r_nk_n}}]$.

As an intermediate step, we now construct mappings $\yw{k_n}$ \tc{(illustrated in Figure \ref{fig:rigCrack}(b))} which have the property that middle parts of the segments $I_i^{(n)}\times \overline{S^{\rm ext}}$ are only subject to a rigid motion, instead of an elastic deformation. The complements of these middle parts contain \tg{no more than $2N_*'+2$} slices, where $N_*':=\lfloor 2\tc{N_{\rm f}}C_{\rm E}/\e_*\rfloor+1$ and $C_{\rm E}$ is a positive constant (independent of $n$ and $\e_*$) that will be introduced in \eqref{eq:ywTransE}. The rigidifying procedure below is presented for an arbitrary but fixed $i\in\{1,\dots,N_{\rm U}\}$.

\begin{figure}[h]
  \caption{Main steps in the proof of Lemma \ref{rigCrack}. Rigid parts of the rod are drawn in grey. (a) The original mapping $\mathscr{y}^{(k_n)}$.  (b) Rigidification of rod segments to construct $\yw{k_n}$. (c) Subsequent shortening of the rigid parts to obtain $\yp{k_n}$.}\label{fig:rigCrack}
  \centering
\includegraphics[width=7.5cm]{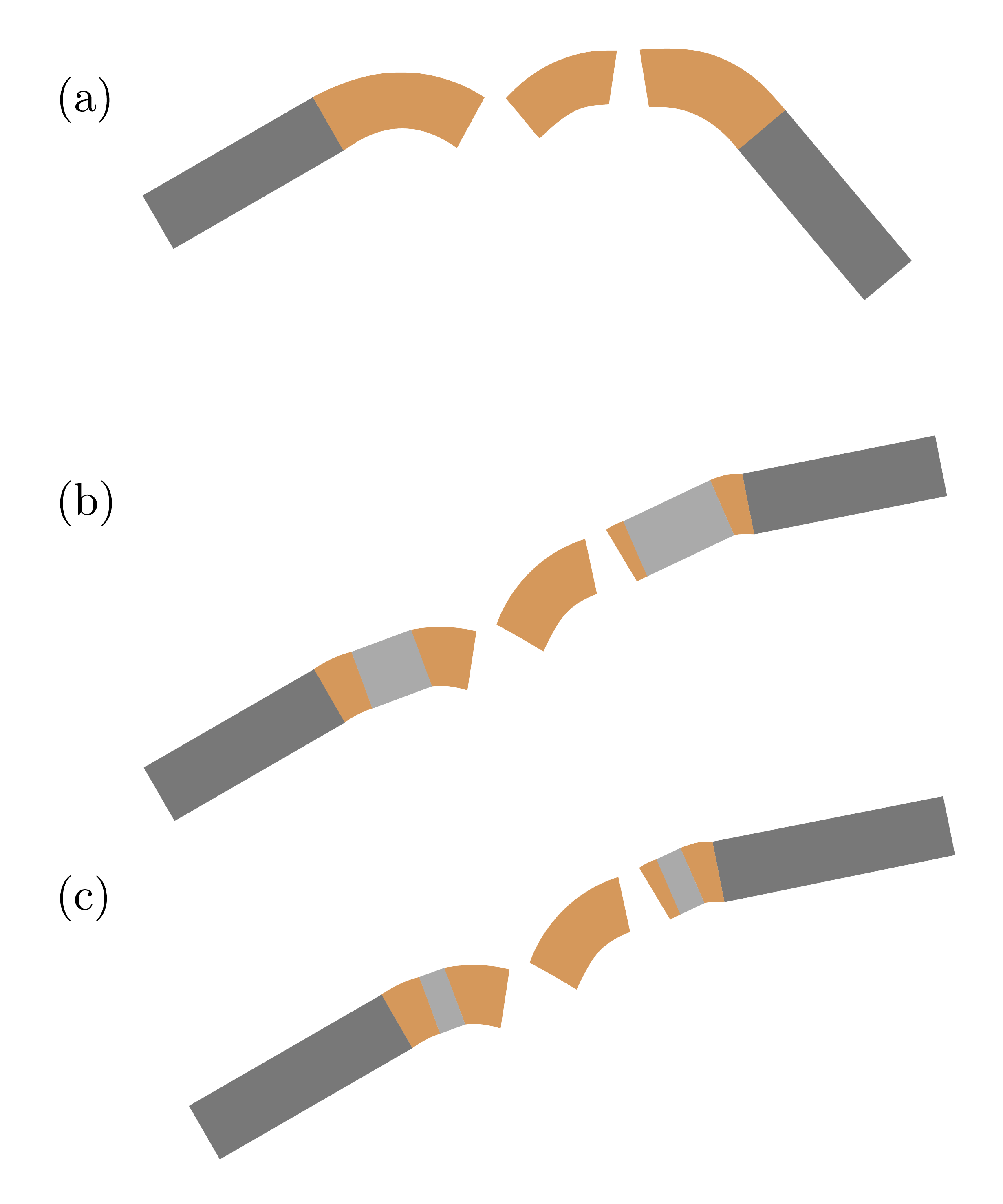}
\end{figure}

\textit{Procedure }(\texttt{R}). As in \cite[Theorem~2.4]{elRods} (which is a reformulation of the compactness theorem in \cite{MMh4}), we get piecewise constant mappings $R^{(k_n)}\colon I_i^{(n)}\goto{\rm SO}(3)$ with discontinuity set contained in $\frac{1}{r_nk_n}\Z$, fulfilling
\begin{align}\label{eq:dist-each-cell}
\begin{split}
r_n&\int_{S_{k_n}(\bar{w}_1)}|\nabla_{r_n,k_n}\mathscr{y}^{(k_n)}-R^{(k_n)}|^2\md w_1\md\xb'\\ 
&\leq \tg{\sum_{m=-1}^1}Cr_n\tg{\int_S\int_{\bar{w}_1+\frac{m}{r_nk_n}}^{\bar{w}_1+\frac{m+1}{r_nk_n}}}{\rm dist}^2(\nabla_{r_n,k_n}\mathscr{y}^{(k_n)},{\rm SO}(3))\md w_1\md\xb'\leq \tg{3}C r_n |S_{k_n}(\bar{w}_1)| (c_{\rm frac}^{(k)})^2 
\leq \frac{C}{k_n^2}
\end{split}
\end{align}
for all $w_1\in[a_i^{(n)},b_i^{(n)})$ by \cite[Theorem~3.1]{FrM02}, growth assumptions on $W_0$, and bounds related to our extension scheme (cf. \eqref{eq:nbhd-rig-est}). Moreover, \cite[Theorem~2.4]{elRods} implies
\begin{align}\label{eq:Rkneib}
\frac{1}{r_nk_n}\Bigl|R^{(k_n)}(w_1)-R^{(k_n)}\bigl(w_1\pm\frac{1}{r_nk_n}\bigr)\Bigr|^2
\leq C\int_{\bigcup\limits_{m=-1}^1 S_{k_n}(\bar{w}_1+\frac{m}{r_nk_n})}{\rm dist}^2(\nabla_{r_n,k_n}\mathscr{y}^{(k_n)},{\rm SO}(3))\md w_1\md\xb'
\end{align}
for all $w_1\in[a_i^{(n)},b_i^{(n)})$.

We now define points that delimit the middle part of $I_i^{(n)}\times \overline{S^{\rm ext}}$ (where $\mathscr{y}^{(k_n)}$ has to be `rigidified') and the sets $W_-^{(n)}$, $W_+^{(n)}$ containing the $w_1$-coordinates of cell midpoints left of or right of this middle part:
\begin{gather*}
a_{0,i}^{(n)}=a_i^{(n)}+\frac{N_*'}{r_nk_n},\; b_{0,i}^{(n)}=b_i^{(n)}-\frac{N_*'}{r_nk_n}\\
W_-^{(n)}=\bigl(\frac{1}{2r_nk_n}+\frac{1}{r_nk_n}\Z\bigr)\cap (a_i^{(n)},a_{0,i}^{(n)})\\
W_+^{(n)}=\bigl(\frac{1}{2r_nk_n}+\frac{1}{r_nk_n}\Z\bigr)\cap (b_{0,i}^{(n)},b_i^{(n)}).
\end{gather*}
The next few steps, till \eqref{eq:ywTransE}, are similar to the proof of the \tc{inequality $\ph\leq\psi$} (cf. Subsection \ref{sec:finCell}), so not all computations will be described in full here. We find $w_-^{(n)}\in W_-^{(n)}$ and $w_+^{(n)}\in W_+^{(n)}$ such that
\begin{gather*}
\tg{\sum_{\ell=-1}^1}\sum_{\xb'\in\mathcal{L}'^{,\rm ext}}k_nW_{\rm tot}^{(k_n)}\bigl(\xb',\bar{\nabla}_{r_n,k_n}\mathscr{y}^{(k_n)}(w_-^{(n)}\tg{+\tfrac{\ell}{r_nk_n}},\xb')\bigr)\leq\frac{\tg{3}}{N_*'}\calE_{k_n}\bigl(\mathscr{y}^{(k_n)},(a_i^{(n)},a_{0,i}^{(n)})\bigr),\\
\tg{\sum_{\ell=-1}^1}\sum_{\xb'\in\mathcal{L}'^{,\rm ext}}k_nW_{\rm tot}^{(k_n)}\bigl(\xb',\bar{\nabla}_{r_n,k_n}\mathscr{y}^{(k_n)}(w_+^{(n)}\tg{+\tfrac{\ell}{r_nk_n}},\xb')\bigr)\leq\frac{\tg{3}}{N_*'}\calE_{k_n}\bigl(\mathscr{y}^{(k_n)},(b_{0,i}^{(n)},b_i^{(n)})\bigr).
\end{gather*}
Writing $R_\pm^{(i,k_n)}$ in place of $R^{(k_n)}(w_\pm^{(n)})$ for short  and using that all the slices centred in $W_\pm^{(n)}$ are intact, from the first inequality in \eqref{eq:dist-each-cell} we get
\begin{equation*}
||\nabla_{r_n,k_n}\mathscr{y}^{(k_n)}-R_\pm^{(i,k_n)}||_{L^2(S_{k_n}(w_\pm^{(n)});\R^{3\times 3})}=O\Bigl(\frac{1}{\sqrt{N_*'r_n}k_n}\Bigr).
\end{equation*}
Choosing vectors $c_-^{(n)}$, $c_+^{(n)}$ as
\begin{gather*}
c_\pm^{(n)}=\dashint_{S_{k_n}(w_\pm^{(n)})}\mathscr{y}^{(k_n)}(w_1,\xb')-R_\pm^{(i,k_n)}\T{\bigl(r_n(w_1-w_\pm^{(n)}),\frac{1}{k_n}\xb'\bigr)}\md w_1\md\xb',
\end{gather*}
we get Poincaré inequalities
\begin{align*}
\MoveEqLeft
\sqrt{\int_{S_{k_n}(w_\pm^{(n)})}|\mathscr{y}^{(k_n)}(w_1,\xb')-R_\pm^{(i,k_n)}\T{\bigl(r_n(w_1-w_\pm^{(n)}),\frac{1}{k_n}\xb'\bigr)}-c_\pm^{(n)}|^2\md w_1\md\xb'}\\
&\leq C\frac{1}{k_n}||\nabla_{r_n,k_n}\mathscr{y}^{(k_n)}-R_\pm^{(i,k_n)}||_{L^2(S_{k_n}(w_\pm^{(n)});\R^{3\times 3})}
\end{align*}
with a constant $C>0$.

With the rotated and shifted version of $\mathscr{y}^{(k_n)}$, given by
\begin{equation}\label{eq:yr}
\mathscr{y}_{\rm r}^{(k_n)}(w_1,\xb'):=R_-^{(i,k_n)}\Bigl[\T{\bigl(R_+^{(i,k_n)}\bigr)}(\mathscr{y}^{(k_n)}(w_1,\xb')-c_+^{(n)})+\begin{pmatrix}r_n(w_+^{(n)}-w_-^{(n)})\\
0\end{pmatrix}\Bigr]+c_-^{(n)},
\end{equation}
set
\begin{equation*}
\yw{k_n}(w_1,\xb')=\begin{cases}
\mathscr{y}^{(k_n)}(w_1,\xb')& a_i^{(n)}-\tg{\tfrac{1}{r_nk_n}} \leq w_1\leq w_-^{(n)}-\frac{1}{2r_nk_n}\\[1.4ex]
\text{pcw. affine (24 simplices/cell)}& w_-^{(n)}-\frac{1}{2r_nk_n}<w_1 <w_-^{(n)}+\frac{1}{2r_nk_n}\\[1.4ex]
R_-^{(i,k_n)}\T{(r_n(w_1-w_-^{(n)}),\frac{1}{k_n}\xb')}+c_-^{(n)}&  w_-^{(n)}+\frac{1}{2r_nk_n}\leq w_1\leq w_+^{(n)}-\frac{1}{2r_nk_n} \\[1.4ex]
\text{pcw. affine (24 simplices/cell)}& w_+^{(n)}-\frac{1}{2r_nk_n}<w_1<w_+^{(n)}+\frac{1}{2r_nk_n}\\[1.4ex]
\mathscr{y}_{\rm r}^{(k_n)}(w_1,\xb') & w_+^{(n)}+\frac{1}{2r_nk_n}<w_1\leq b_i^{(n)}+\tg{\tfrac{1}{r_nk_n}}
\end{cases}
\end{equation*}
so that $\yw{k_n}$ is defined on $I_i^{(n)}\times \overline{S^{\rm ext}}$. Besides, to prepare future rigidification on possible next intervals, we redefine $\mathscr{y}^{(k_n)}$ by $\mathscr{y}^{(k_n)}:=\mathscr{y}_{\rm r}^{(k_n)}$ on $[b_i^{(n)}\tg{+\frac{1}{r_nk_n}},1]\times\overline{S^{\rm ext}}$.

After some calculations we deduce that on any atomic cell $Q$ such that ${\rm Int}\,Q\subset S_{k_n}(w_-^{(k_n)})$,
\begin{align*}
\bigl|\bar{\nabla}_{r_n,k_n}\mathscr{y}^{(k_n)}|_Q-R_-^{(i,k_n)}\bar{\Id}\bigr|&=O\Bigl(\frac{1}{\sqrt{N_*'k_n}}\Bigr)\quad\text{and consequently,}\\
\bigl|\bar{\nabla}_{r_n,k_n}\yw{k_n}|_Q-R_-^{(i,k_n)}\bar{\Id}\bigr|&=O\Bigl(\frac{1}{\sqrt{N_*'k_n}}\Bigr),
\end{align*}
which implies that for all $n$ sufficiently large, the energetic error occurring on the transition slice $S_{k_n}(w_-^{(k_n)})$ is controlled by our choice of $N_*'$:
\begin{equation}\label{eq:ywTransE}
\bigl|\calE_{k_n}\bigl(\mathscr{y}^{(k_n)},w_-^{(n)}+\frac{1}{2r_nk_n}(-1,1)\bigr)-\calE_{k_n}\bigl(\yw{k_n},w_-^{(n)}+\frac{1}{2r_nk_n}(-1,1)\bigr)\bigr|\leq\frac{C_{\rm E}}{N_*'}.
\end{equation}
It should be stressed that the constant $C_{\rm E}$ above does not depend on $n$ or $\e_*$. Due to the definition of $\mathscr{y}_{\rm r}^{(k_n)}$, an analogous computation reveals that \eqref{eq:ywTransE} also holds if $w_-^{(n)}$ is replaced with $w_+^{(n)}$.

Later we will have to check that $(\yw{k_n})_{n=1}^\infty$ is an admissible competitor of $(\mathscr{y}^{(k_n)})_{n=1}^\infty$ in the cell formula. Therefore we now show that the error incurred by the boundary condition due to the previous steps of Procedure (\texttt{R}) tends to zero.

By our interpolation scheme, on any atomic cell $Q$ contained in $I_i^{(n)}\times\overline{S^{\rm ext}}$ we have (cf. \cite[Lemma~3.5]{BS09})
\begin{align*}
\bigl\lVert\nabla_{r_n,k_n}\mathscr{y}^{(k_n)}|_Q\bigr\rVert_\infty 
\leq 24 \dashint_Q |\nabla_{r_n,k_n}\mathscr{y}^{(k_n)}|\md w_1\md\xb'
\leq C \bigl|\,\bar{\nabla}_{r_n,k_n}\mathscr{y}^{(k_n)}|_Q\,\bigr| 
\leq C 
\end{align*}
since ${\rm dist}^2(\bar{\nabla}_{r_n,k_n}\mathscr{y}^{(k_n)},\bar{\rm SO}(3))\leq (c_{\rm frac}^{(k_n)})^2$. This proves that the mappings $\mathscr{y}^{(k_n)}|_{I_i^{(n)}\times\overline{S^{\rm ext}}}$ are Lipschitz with the uniform constant $Cr_n$. In particular, 
\begin{equation*}
\lim_{n\goto\infty}|c_+^{(n)}-c_-^{(n)}|=0.
\end{equation*}

Since by iterating \eqref{eq:Rkneib} we derive a `pointwise curvature estimate' (as in \cite{MMh4,FrM02})
\begin{equation*}
|R_+^{(i,k_n)}-R_-^{(i,k_n)}|^2\leq Cr_n^2k_n^2\int_{I_i^{(n)}\times S}{\rm dist}^2(\nabla_{r_n,k_n}\mathscr{y}^{(k_n)},{\rm SO}(3))\md w_1\md\xb'=O(r_n)
\end{equation*}
we obtain for $\mathscr{y}_{\rm r}^{(k_n)}$ from \eqref{eq:yr} that $|\mathscr{y}_{\rm r}^{(k_n)}-\mathscr{y}^{(k_n)}|\goto 0$ uniformly.

This finishes Procedure (\texttt{R}) for the chosen $i$.

We construct $\yw{k_n}$ by letting $\yw{k_n}(w_1,\xb'):=\mathscr{y}^{(k_n)}(w_1,\xb')$ for every $-1\leq w_1\leq a_1^{(n)}\tg{-\frac{1}{r_nk_n}}$ and $\xb'\in \overline{S^{\rm ext}}$ and then by successively applying Procedure (\texttt{R}) for $i=1,2,\dots N_{\rm U}$ (it should be kept in mind that after each invocation of Procedure (\texttt{R}), $\mathscr{y}^{(k_n)}$ is redefined on $[b_i^{(n)}\tg{+\frac{1}{r_nk_n}},1]\times\overline{S^{\rm ext}}$ so that in step $i+1$ we get the modified mapping $\mathscr{y}^{(k_n)}$ from step $i$ as input).

On $(\frac{1}{r_nk_n}\lfloor\frac{3}{4}r_nk_n\rfloor,1]\times\overline{S^{\rm ext}}$, we define $\yw{k_n}$ as $\yw{k_n}:=\mathscr{y}^{(k_n)}$, where $\mathscr{y}^{(k_n)}$ is understood as the transformed mapping after the $N_{\rm U}$-th step of rigidification.

As we have seen above, the affine transformations given by \eqref{eq:yr} at each step vanish in the limit. Hence, $((r_n)_{n=1}^\infty,(k_n)_{n=1}^\infty,(\yw{k_n})_{n=1}^\infty)\in\mathcal{V}_{\yb^+-\yb^-,(R^-)^{-1}R^+}$.

To summarize, the sequence $(\yw{k_n})_{n=1}^\infty$ satisfies
\begin{align*}
\ph\bigl(\yb^+-\yb^-,(R^-)^{-1}R^+\bigr)
&\leq\limsup_{n\goto\infty}\calE_{k_n}\bigl(\yw{k_n},[-1,1]\bigr)\\
&\leq\ph\bigl(\yb^+-\yb^-,(R^-)^{-1}R^+\bigr)+2N_{\rm U}\frac{C_{\rm E}}{N_*'}\leq \ph\bigl(\yb^+-\yb^-,(R^-)^{-1}R^+\bigr)+\e_*.
\end{align*}

Now we proceed to construct the modifications $\yp{k_n}$ of $\yw{k_n}$ which will have more localized non-rigid parts (as depicted in Figure \ref{fig:rigCrack}(c)).
 
Since no confusion arises, we again use $R_\pm^{(k_n)}$ and $y_\pm^{(k_n)}$ to denote the rigid deformations near the interval boundaries, i.e.
\begin{align*}
\yw{(k_n)}(w_1,\xb')=R_\pm^{(k_n)}\T{\bigl(r_nw_1,\frac{1}{k_n}\xb'\bigr)}+y_\pm^{(k_n)}\text{ on }I^\pm\times \overline{S^{\rm ext}}.
\end{align*}
Now we first extend $\yw{k_n}$ rigidly to a function on $\R\times\overline{S^{\rm ext}}$ by requiring this formula to hold true on $(-\infty,-\frac{3}{4})\times \overline{S^{\rm ext}}$ and $(\frac{3}{4},\infty)\times \overline{S^{\rm ext}}$, with the obvious interpretation of the $\pm$ sign.

If $j=j_i$ for some $i\in\{1,2,\dots,N_{\rm U}\}$, then we write $w_-^{(i,n)}$, $w_+^{(i,n)}$ in place of $w_-^{(n)}$, $w_+^{(n)}$ from Procedure (\texttt{R}), respectively, to stress the dependence on $i$. We set $d^{(i,n)}= w_{\rm +}^{(i,n)}-w_{\rm -}^{(i,n)}-\frac{1}{r_nk_n}$ and also recall the definition of $R_-^{(i,k_n)}$ on this interval. Now consecutively do the following steps for $i\in\{1,2,\dots,N_{\rm U}\}$, in reverse order starting with $i=N_{\rm U}$: 
\begin{align*}
\yp{k_n}(w_1,\xb')&:=\begin{cases}
\yw{k_n}(w_1,\xb') & w_1 \leq w_{\rm -}^{(i,n)}+\frac{1}{2r_nk_n},\\
\yw{k_n}(w_1+d^{(i,n)},\xb')-r_nd^{(i,n)} R_-^{(i,k_n)} e_1
&w_1 > w_{\rm -}^{(i,n)}+\frac{1}{2r_nk_n},
\end{cases}\\
\tc{\yw{k_n}(w_1,\xb')}&:=\tc{\yp{k_n}(w_1,\xb'),\quad\qquad\qquad\qquad w_1\geq w_-^{(i,n)}+\tfrac{1}{2r_nk_n},\; \xb'\in\overline{S^{\rm ext}}.}
\end{align*}
This finally results in a configuration with 
\begin{equation*}
\yp{k_n}(w_1,\xb')
=\yw{k_n}(w_1,\xb')
=R_-^{(k_n)}\T{\bigl(r_nw_1,\frac{1}{k_n}\xb'\bigr)}+y_-^{(k_n)}
\end{equation*}
if $w_1 \leq -\frac{3}{4}$, \tc{$\xb'\in\overline{S^{\rm ext}}$}, and 
\begin{equation*}
\yp{k_n}(w_1,\xb')=
\yw{k_n}(w_1+d^{(n)},\xb')-r_n c^{(n)} 
=R_+^{(k_n)}\T{\bigl(r_nw_1,\frac{1}{k_n}\xb'\bigr)}+r_nd^{(n)}R_+^{(k_n)}e_1+y_+^{(k_n)}-r_n c^{(n)}
\end{equation*}
where $d^{(n)}=\sum_{i=1}^{N_{\rm U}}d^{(i,n)}$ and $c^{(n)}=\sum_{i=1}^{N_{\rm U}} d^{(i,n)} R_-^{(i,k_n)} e_1$, if $w_1 \geq \frac{3}{4} - d^{(n)}$ \tc{and $\xb'\in\overline{S^{\rm ext}}$}.

Observe that $\calE_{k_n}(\yp{k_n},[-1,1])=\calE_{k_n}(\yw{k_n},[-1,1])$ for every $n\in\N$ as we have only shortened the intermediate rigid parts. Also, \tc{the length of the non-rigid part now satisfies}
\begin{equation*}
\tc{\frac{1}{r_nk_n}\Bigl\lfloor\frac{3}{4}\Bigr\rfloor - d^{(n)} - \frac{1}{r_nk_n}\Bigl(- \Bigl\lfloor\frac{3}{4}\Bigr\rfloor+1\Bigr)}
\le \frac{1}{r_nk_n} \bigl( (2 N_*'\tg{+4})\tc{( N_{\rm f}+1)} + N_{\rm f} \bigr).
\end{equation*}
Setting $N_* = (2 N_*\tg{+4})( N_{\rm f}+1)+N_{\rm f}$ and shifting we finally obtain $\yp{k_n}$ as claimed.
\end{proof}
\begin{rem}
Lemma \ref{rigCrack} shows that the choice of $I^\pm$ in the definition of $\ph$ was arbitrary and that a different positive length of $I^\pm$ which still leaves a nonempty middle interval for fracture would give the same value of $\ph$. 
\end{rem}

Our next task is to prove that the passages to subsequences $(k_n)$ can be avoided when approximating the value of the cell formula.
\begin{prop}\label{arbAtomDist}
Suppose that $\yb^-,\yb^+\in\R^3$ and $R^-,R^+\in{\rm SO}(3)$. Then for any $\e_*>0$ and any nonincreasing sequence $\{\rho_k\}_{k=1}^\infty\subset (0,\infty)$ with $\lim_{k\goto\infty}\rho_k =0$ and $\lim_{k\goto\infty}\rho_k k =\infty$ there exist deformations $\bar{\mathscr{y}}^{(k)}\colon ([-1,1]\times \overline{S^{\rm ext}})\goto\R^3$ such that $((\rho_k)_{k=1}^\infty,(k)_{k=1}^\infty,(\bar{\mathscr{y}}^{(k)})_{k=1}^\infty)\in\mathcal{V}_{\yb^+-\yb^-,(R^-)^{-1}R^+}$ and
\begin{equation*}
\limsup_{k\goto\infty}\calE_k(\bar{\mathscr{y}}^{(k)},[-1,1])\leq\ph\bigl(\yb^+-\yb^-,(R^-)^{-1}R^+\bigr)+\e_*.
\end{equation*}
\end{prop}
\begin{proof}
For a given $\e_*>0$ we choose $N_*\in\N$, a (without loss of generality nondecreasing) sequence $(k_n)_{n=1}^\infty$, and mappings $\yp{k_{n}}\in\mathrm{PAff}(\Lambda_{r_{n},k_{n}})$ as in Lemma~\ref{rigCrack} so that 
\begin{align*}
\limsup_{n\goto\infty}\calE_{k_{n}}(\yp{k_{n}},[-1,1])\leq &\ph\bigl(\yb^+-\yb^-,(R^-)^{-1}R^+\bigr)+\e_*,
\end{align*}
and, for suitable $\ypp{\pm}{k_{n}}\in\R^3,\; \Rp{\pm}{k_{n}}\in{\rm SO}(3)$ with $\ypp{\pm}{k_{n}}\goto \yb^\pm$, $\Rp{\pm}{k_{n}}\goto R^\pm$, after a rigid extension to the left and to the right, 
\begin{align*}
\yp{k_{n}}(w_1,\xb')=\Rp{\pm}{k_{n}}\T{\bigl(r_{n}w_1,\frac{\xb'}{k_{n}}\bigr)}&+\ypp{\pm}{k_{n}}\text{ on }\bigl(\R\setminus I_{\rm c}^{(n)}\bigr)\times \overline{S^{\rm ext}}
\end{align*}
where $I_{\rm c}^{(n)}=\frac{1}{r_{n}k_{n}}[-N_*,N_*]$.

For each $k\in\N$ find $n_k\in\N$ such that $k_{n_k}^{-1}\leq k^{-1}\leq k_{n_k-1}^{-1}$. Set 
\begin{align*}
\bar{\mathscr{y}}^{(k)}(w_1,\xb')
:=\frac{k_{n_k}}{k}\yp{k_{n_k}}\bigl(\frac{\rho_kk}{r_{n_k}k_{n_k}} w_1, \xb'\bigr) ,\quad(w_1,\xb')\in[-1,1]\times \overline{S^{\rm ext}}.
\end{align*}
Like this, $\bar{\mathscr{y}}^{(k)}$ is well-defined (as far as the boundary condition on $I^\pm\times \overline{S^{\rm ext}}$ is concerned), at worst for all $k$ larger than a certain $\bar{k}\in\N$.  If it is the case that $\bar{k}>1$, we define $\mathscr{y}^{(1)},\dots,\mathscr{y}^{(\bar{k}-1)}$ as we like, e.g.\ by extending the boundary rigid motions to all of $[-1,1]\times \overline{S^{\rm ext}}$. Then for $k\geq\bar{k}$,
\begin{equation*}
\bar{\nabla}_{\rho_k,k}\mathscr{y}^{(k)}(w_1,\xb')
=\bar{\nabla}_{r_{n_k},k_{n_k}}\yp{k_{n_k}}\bigl(\frac{\rho_kk}{r_{n_k}k_{n_k}} w_1, \xb'\bigr) 
\end{equation*}
and 
\begin{equation*}
kW_{\rm tot}^{(k)}\Bigl(\xb', \bar{\nabla}_{\rho_k,k}\mathscr{y}^{(k)}(w_1,\xb') \Bigr)
\leq \tc{k_{n_k}W_{\rm tot}^{(k_{n_k})}\Bigl(\xb',}\bar{\nabla}_{r_{n_k},k_{n_k}}\yp{k_{n_k}}\bigl(\frac{\rho_kk}{r_{n_k}k_{n_k}} w_1, \xb'\bigr) \Bigr)
\end{equation*}
by assumption (W4) on the cell energy. This yields
\begin{align*}
\ph\bigl(\yb^+-\yb^-,(R^-)^{-1}R^+\bigr)
&\leq\limsup_{k\goto\infty}\calE_k(\bar{\mathscr{y}}^{(k)},[-1,1])\\
&\leq\limsup_{k\goto\infty}\calE_{k_{n_k}}(\yp{k_{n_k}},[-1,1])
\leq \ph\bigl(\yb^+-\yb^-,(R^-)^{-1}R^+\bigr)+\e_*.\qedhere
\end{align*}
\end{proof}

The approximating sequence $(\mathscr{y}^{(k)})$ around crack points can be chosen to be bounded in $L^\infty$ in a universal way – this is the content of

\begin{prop}\label{LinftyBdd}
Suppose that $\yb^-,\yb^+\in\R^3$, $R^-,R^+\in{\rm SO}(3)$ and $(r_k)_{k=1}^\infty\subset (0,\infty)$ is a nonincreasing sequence with $\lim_{k\goto\infty}r_k =0$ and $\lim_{k\goto\infty}r_k k =\infty$. Assume that $\mathscr{y}^{(k)}\colon ([-1,1]\times \overline{S^{\rm ext}})\goto\R^3$ is such that $((r_k)_{k=1}^\infty,(k)_{k=1}^\infty,(\mathscr{y}^{(k)})_{k=1}^\infty)\in\mathcal{V}_{\yb^+-\yb^-,(R^-)^{-1}R^+}$ with 
\begin{equation*}
 \mathscr{y}^{(k)}(w_1,\xb')=R_\pm^{(k)}\T{\bigl(r_nw_1,\frac{1}{k}\xb'\bigr)}+y_\pm^{(k)}\text{ on }I^\pm\times \overline{S^{\rm ext}} 
\end{equation*}
for $R_\pm^{(k)}\goto R^\pm$, $y_\pm^{(k)}\goto \yb^\pm$. If the maximum interaction range property \tc{(W9)} with rate $(M_k)_{k=1}^{\infty}$ holds true, then there exists a modification $\bar{\mathscr{y}}^{(k)}$ with $((r_k)_{k=1}^\infty,(k)_{k=1}^\infty,(\bar{\mathscr{y}}^{(k)})_{k=1}^\infty)\in\mathcal{V}_{\yb^+-\yb^-,(R^-)^{-1}R^+}$ such that 
\begin{equation*}
|\calE_k(\bar{\mathscr{y}}^{(k)},[-1,1])
-\calE_k(\mathscr{y}^{(k)},[-1,1])|
\leq \frac{C}{kM_k}\calE_k(\mathscr{y}^{(k)},[-1,1]), 
\end{equation*}
$\bar{\mathscr{y}}^{(k)}=\mathscr{y}^{(k)}$ on $(I^- \cup I^+)\times \overline{S^{\rm ext}}$ and 
\begin{equation*}
||\mathrm{dist}(\bar{\mathscr{y}}^{(k)}, \{y_-^{(k)},y_+^{(k)}\})||_\infty\leq Cr_kM_kk\calE_k(\mathscr{y}^{(k)},[-1,1]). 
\end{equation*}
\end{prop}
\begin{proof}
We write $\tc{D}(\bar{x}) = \bar{x} + \{(\frac{1}{r_kk}\zf^i_1,(\zf^i)');\; i=1,\ldots,8\}$ for the corners of the cell with midpoint $\bar{x}\in\Lambda_{r_k,k}'$. Our strategy is to move back all pieces of the rod that are too far from $\{y_-^{(k)},y_+^{(k)}\}$. Fix $k\in\N$ and consider the undirected graph $\mathfrak{G}=(\mathfrak{V},\mathfrak{E})$, where $\mathfrak{V}=\Lambda_{r_k,k}$ and 
\begin{equation*}
\{x,x^\dagger\}\in\mathfrak{E}\Leftrightarrow(\exists\, \bar{x}\in\Lambda_{r_k,k}'\colon x,x^\dagger\in \tc{D}(\bar{x})\wedge |\mathscr{y}^{(k)}(x)-\mathscr{y}^{(k)}(x^\dagger)|< M_k).
\end{equation*}
Let $\mathfrak{G}_1,\mathfrak{G}_2,\dots,\mathfrak{G}_{n_{\rm G}}$ be the connected components of $\mathfrak{G}$, numbered \tc{in such a way} that $(I^-\times \overline{S^{\rm ext}}) \cap \Lambda_{r_k,k} \in \mathfrak{G}_1$ and $(I^+\times \overline{S^{\rm ext}}) \cap \Lambda_{r_k,k} \in \mathfrak{G}_{n_{\rm G}}$. Accordingly we partition $\{\zf^1,\zf^2,\dots,\zf^8\}=Z_1(\bar{x})\dot{\cup}Z_2(\bar{x})\dot{\cup}\cdots\dot{\cup}Z_{n_{\bar{x}}}(\bar{x})$ for every $\bar{x}\in\Lambda_{r_k,k}'$, where $Z_i(\bar{x})\neq\emptyset$, so that $\zf^j,\zf^m\in Z_\ell(\bar{x})$ for some $\ell\in\{1,2,\dots,n_{\bar{x}}\}$ if and only if there is $i_{\rm V}\in\{1,2,\dots,n_{\rm G}\}$ such that $\bar{x}+\frac{1}{k}\zf^j,\bar{x}+\frac{1}{k}\zf^m\in\mathfrak{V}_{i_{\rm V}}$, the set of vertices of $\mathfrak{G}_{i_{\rm V}}$. Then the induced components of atomic cells are far apart: for any $\bar{x}\in\Lambda_{r_k,k}'$ and $1\leq i<j\leq n_{\bar{x}}$, we have $\mathrm{dist}(y^{(k)}(\bar{x}+Z_i(\bar{x})),y^{(k)}(\bar{x}+Z_j(\bar{x})))\geq M_k$. 

Similarly as before we observe that the number of atomic cells `broken' by $\mathscr{y}^{(k)}$ is controlled by the energy so that the number $n_{\rm G}$ of connected components of $\mathfrak{G}$ satisfies a bound of the form 
\begin{equation*}
 n_{\rm G}\leq C_1 \calE_k(\mathscr{y}^{(k)},[-1,1]) 
\end{equation*}
with a constant $C_1>0$. 
 The construction further implies that the diameter of each component after deformation is bounded by
\begin{equation*}
 \mathrm{diam}\mathscr{y}^{(k)}(\tc{\mathfrak{V}_{i}}) 
 \leq C_2 M_k r_kk,\quad i=1,\ldots,n_{\rm G},
\end{equation*}
with \tc{another} constant $\tc{C_2}>0$.

For the first and last component we have \begin{equation*}
\mathrm{dist}(\mathscr{y}^{(k)}(\mathfrak{V}_1), \{y_-^{(k)}\}) \le \tc{C_3}M_kr_kk\quad\text{and}\quad\mathrm{dist}(\mathscr{y}^{(k)}(\mathfrak{V}_{n_{\rm G}}), \{y_+^{(k)}\}) \le \tc{C_3}M_kr_kk. 
\end{equation*}
If $n_{\rm G} \ge 3$, we can shift graph components $\mathfrak{G}_i$, $i=\tc{2},\ldots,n_{\rm G}-1$, without \tc{considerably} changing the total energy, provided we do not put the components at a distance less than $M_k$. Specifically, for $\tc{\gamma}=2M_k+(C_2+\tc{C_3}) M_k r_kk\le (2+C_2+\tc{C_3}) M_k r_kk$ and $|e|=1$ with $e \perp y_+^{(k)}-y_-^{(k)}$ the points $y_-^{(k)}+\tc{(i-1)}\gamma e$, $i=\tc{2},\ldots,n_{\rm G}-1$, have a distance $\ge\gamma$ from each other and from $\{y_+^{(k)},y_-^{(k)}\}$. We then define $\bar{\mathscr{y}}^{(k)}$ by shifting $\mathfrak{G}_i$ rigidly in such a way that $y_-^{(k)}+\tc{(i-1)}\gamma e \in \bar{\mathscr{y}}^{(k)}(\mathfrak{V}_i)$, $i=\tc{2},\ldots,n_{\rm G}-1$. 

Then indeed the shifted components have the required minimal distances and moreover 
\begin{equation*}
\mathrm{dist}(\mathscr{y}^{(k)}(\mathfrak{V}_i), \{y_-^{(k)}\}) \le n_{\rm G} \gamma \le C_1 \calE_k(\mathscr{y}^{(k)},[-1,1]) (2+C_2+\tc{C_3}) M_k r_kk,  
\end{equation*}
$i=\tc{2},\ldots,n_{\rm G}-1$. The assertion follows now by noting that $\bar{\mathscr{y}}^{(k)} = \mathscr{y}^{(k)}$ on $\mathfrak{V}_1 \cup \mathfrak{V}_{n_{\rm G}}$ \tc{and} 
\begin{equation*}
|\calE_k(\bar{\mathscr{y}}^{(k)},[-1,1])
-\calE_k(\mathscr{y}^{(k)},[-1,1])|
\leq \tc{C} \calE_k(\mathscr{y}^{(k)},[-1,1]) \frac{\tc{C_{\rm far}}}{kM_k},  
\end{equation*}
as only broken cells have been altered.
\end{proof}

\subsection{Construction of recovery sequences}
\begin{proof}[Proof of Theorem \ref{Gamma}(ii)]
It is known from the theory of $\Gamma$-convergence that for any $\e>0$ it suffices to find a recovery sequence with $\limsup_{k\goto\infty} kE^{(k)}(y^{(k)})\leq E_{\rm lim}(\yb,d_2,d_3)+\e$, which is trivial if $(\yb,d_2,d_3)\not\in\calA$. In the case that $(\yb,d_2,d_3)\in\calA$, let $(\sigma^i)_{i=0}^{\bar{n}_{\rm f}+1}$ be the partition of $[0,L]$ such that $\{\sigma^i\}_{i=1}^{\bar{n}_{\rm f}}=S_{\yb}\cup S_R$, where $S_R:=S_{\yb'}\cup S_{d_2}\cup S_{d_3}$. Depending on the assumptions on $\yb$, $d_2$, $d_3$, we treat two different cases separately.

First additionally suppose that $\yb|_{(\sigma^{i-1},\sigma^i)}\in\mathcal{C}^3((\sigma^{i-1},\sigma^i);\R^3)$, $\tc{d_s}|_{(\sigma^{i-1},\sigma^i)}\in \mathcal{C}^2((\sigma^{i-1},\sigma^i);\R^3)$, $s=2,3$, for all $i\in\{1,2,\dots,\bar{n}_{\rm f}+1\}$ and that $R=(\partial_1 \yb| d_2 | d_3)$ is constant on the sets $(\sigma^0,\sigma^0+\eta)$, $(\sigma^i-\eta, \sigma^i)$, $(\sigma^i,\sigma^i+\eta)$, $i\in\{1,2,\dots,\bar{n}_{\rm f}\}$, and $(\sigma^{\bar{n}_{\rm f}}-\eta,\sigma^{\bar{n}_{\rm f}})$ for some $\eta>0$. If $k\in\N$, write $I_0^k:=[-\frac{1}{k},\frac{1}{k}\lfloor k\sigma^1\rfloor]$, $I_i^k:=[\frac{1}{k}\lfloor k\sigma^i\rfloor+\frac{1}{k},\frac{1}{k}\lfloor k\sigma^{i+1}\rfloor]$ for $i=1,2,\dots,\bar{n}_{\rm f}-1$ and $I_{\bar{n}_{\rm f}}^k:=[\frac{1}{k}\lfloor k\sigma^{\bar{n}_{\rm f}}\rfloor+\frac{1}{k},L_k+\frac{1}{k}]$. 

Our analysis of elastic rods in \cite[Section~3.4]{elRods} shows that for a suitable choice of $\b(\cdot,x')\in\calC^1([0,L];\R^3)$ for each $x'\in\mathcal{L}^{\rm ext}$ and of $q\in\calC^2([0,L];\R^3)$, by setting 
\begin{equation}\label{eq:recov-seq-first}
\ybk(\xb):=\yb(\xb_1)+\frac{1}{k}\xb_2 d_2(\xb_1)+\frac{1}{k}\xb_3 d_3(\xb_1)+\frac{1}{k}q(\xb_1)+\frac{1}{k^2}\b(\xb),\quad\xb\in\{0,\tfrac{1}{k},\ldots,L_k\}\times\mathcal{L}^{\rm ext},
\end{equation}
appropriately extended and interpolated on $[-\frac{1}{k},\ldots,L_k+\tfrac{1}{k}]\times\overline{S^{\rm ext}}$, one has $\ybk\goto\yb$ in $L^2$ on $(0,L)\times S^{\rm ext}$ as well as 
\begin{equation*}
\sum_{x\in\{-\frac{1}{2k},L_k+\frac{1}{2k}\}\times\mathcal{L}'^{,\rm ext}} \tc{k}W_{\rm end}^{\tc{(k)}}\big(x_1,x',\bar{\nabla}_k\ybk(\xb)\big) 
\goto 0
\end{equation*}
and 
\begin{align}
\MoveEqLeft k\int_{I_i^k\times \overline{S^{\rm ext}}} W_{\rm tot}^{(k)}(\xb',\bar{\nabla}_k\ybk)\md\xb\nonumber\\
&\goto \frac{1}{2} \int_{\sigma^i}^{\sigma^{i+1}} \int_{S^{\rm ext}} Q_{\rm tot}\biggl(\xb', \T{R}(x_1) \Bigl(\frac{\pl R}{\pl\xb_1}(x_1)\T{(0, \bar{\xb}_2, \bar{\xb}_3)} + \frac{\pl q}{\pl\xb_1}(x_1) \Bigr) \T{e_1}\bar{\Id} \nonumber \\ 
&\qquad\qquad\qquad\qquad\quad + \T{R}(x_1) \frac{\pl R}{\pl\xb_1}(x_1) \bigl[\zf_1^i\T{(0, \zf_2^i,\zf_3^i)} \bigr]_{i=1}^8
+ \T{R}(x_1)\bigl(\d2d\b(\xb)|\d2d\b(\xb)\bigr)\biggr)\md\xb\label{eq:limsupPartConcl}\\
&\leq \frac{1}{2}\int_{\sigma^i}^{\sigma^{i+1}} \int_{S^{\rm ext}} Q_3^{\rm rel}\Bigl(\T{R}(\xb_1)\frac{\pl R}{\pl\xb_1}(\xb_1)\Bigr)\md\xb_1 + \e.\label{eq:est-Ek-elast} 
\end{align}
Indeed one can choose $\beta \equiv 0$ and $q \equiv 0$ on $(\sigma^{i},\sigma^{i}+\frac{\eta}{2})\cup(\sigma^{i+1}-\frac{\eta}{2},\sigma^{i+1})$ as $R$ by assumption is constant on a neighbourhood of these sets. So we have 
\begin{align*}
\ybk(x) 
= \begin{cases} 
\yb(\sigma^i+) + R(\sigma^i+)\tc{\T{(x_1 - \sigma^i,\xb')}} 
&\text{for } \xb_1\in(\sigma^i,\sigma^i+\tfrac{\eta}{2}), \\ 
\yb(\sigma^{i+1}-) + R(\sigma^{i+1}-)\tc{\T{(x_1 - \sigma^{i+1},\xb')}} 
&\text{for } \xb_1\in(\sigma^{i+1}-\tfrac{\eta}{2},\sigma^{i+1}).
  \end{cases}
\end{align*}

We now update $\ybk$ by replacing portions near the jumps $\sigma^i$ (and matching all parts by applying suitable rigid motions). Fix a sequence $(r_k)_{k=1}^{\infty}$ such that $r_k\goto 0$ and $r_kk\to \infty$. By Proposition~\ref{arbAtomDist} for each $i=1,\ldots,\bar{n}_{\rm f}$ we can choose $\mathscr{y}^{(k)}_i\colon ([-1,1]\times \overline{S^{\rm ext}})\goto\R^3$ such that $((r_k)_{k=1}^\infty,(k)_{k=1}^\infty,(\mathscr{y}^{(k)})_{k=1}^\infty)\in\mathcal{V}_{\yb(\sigma^i+)-\yb(\sigma^i-),(R(\sigma^i-))^{-1}R(\sigma^i+)}$ with 
\begin{equation*}
 \mathscr{y}^{(k)}(w_1,\xb')=R_\pm^{(k,i)}\T{\bigl(r_nw_1,\frac{1}{k}\xb'\bigr)}+y_\pm^{(k,i)}\text{ on }I^\pm\times \overline{S^{\rm ext}} 
\end{equation*}
for $R_\pm^{(k,i)}\goto R(\sigma^i\pm)$, $y_\pm^{(k,i)}\goto \yb^\pm$ which satisfies the energy estimate 
\begin{align}\label{eq:est-Ek-crack}
\limsup_{k\goto\infty}\calE_k\bigl(\mathscr{y}^{(k)}_i,[-1,1]\bigr)\leq \ph(\yb(\sigma^i+)-\yb(\sigma^i-),R(\sigma^i-)^{-1}R(\sigma^i+))+\e. \end{align}

\tc{Let $H_{\sigma,r}(\xb):=(\frac{1}{r}(\xb_1-\sigma),\xb')$ for any $r>0$.}  Noticing that $\yb^{(k)}$ is rigid near a jump as are the $\mathscr{y}^{(k)}_i$ near $\pm1$, we can now define a modification $\ybk_{\rm tot}$ of $\ybk$ by setting  
\begin{equation*}
\yb_{\rm tot}^{(k)}(\xb)=\begin{cases}
\yb^{(k)}(\xb) & -\frac{1}{k}\leq \xb_1\leq \tc{\sigma^1_{k}-r_k},  \\
O_-^{(k,i)}\mathscr{y}_i^{(k)}\circ H_{\sigma^i_{k},r_k}(\xb)+c_-^{(k,i)} & \sigma^i_{k}-r_k<\xb_1\leq \sigma^i_{k}+r_k,\;i=1,\ldots,\bar{n}_{\rm f},\\
O_+^{(k,i)}\yb^{(k)}(\xb)+c_+^{(k,i)} & \sigma^i_{k}+r_k <\xb_1 \le \sigma^{i+1}_k-r_k,\;i=1,\ldots,\bar{n}_{\rm f}-1, \\ 
O_+^{(k,\bar{n}_{\rm f})}\yb^{(k)}(\xb)+c_+^{(k,\bar{n}_{\rm f})} & \sigma^{\bar{n}_{\rm f}}_{k}+r_k <\xb_1 \le L_k+\frac{1}{k}, 
\end{cases}
\end{equation*}
where \tc{$O_\pm^{(k,i)} \in {\rm SO}(3)$ and $c_\pm^{(k,i)} \in\R^3$} are such that 
\begin{equation*}
O_-^{(k,i)}\mathscr{y}_{i}^{(k)}\circ H_{\sigma^{i}_{k},r_k}+c_-^{(k,i)}
=\begin{cases}
O_+^{(k,i-1)}\yb^{(k)}+c_+^{(k,i-1)} 
& \text{on }(\sigma^{i}_{k}-r_k,\sigma^{i}_{k}-\frac{3}{4}r_k) \times S^{\rm ext}, \\ 
O_+^{(k,i)}\yb^{(k)}+c_+^{(k,i)} 
& \text{on }(\sigma^{i}_{k}+\frac{3}{4}r_k,\sigma^{i}_{k}+r_k) \times S^{\rm ext}
 \end{cases}
\end{equation*}
for $i=1,\ldots,\bar{n}_{\rm f}$ (and we have set $O_+^{(k,0)}:=\Id$, $c_+^{(k,0)}:=0$). Since $R_\pm^{(k,i)}\goto R(\sigma^i\pm)$, $y_\pm^{(k,i)}\goto \yb^\pm$ we get $O_\pm^{(k,i)} \goto \Id$ and $c_\pm^{(k,i)} \goto 0$ as $k\goto\infty$. Thus we still have $\ybk_{\rm tot}\goto\yb$ in $L^2((0,L)\times S^{\rm ext})$. By \eqref{eq:est-Ek-elast} and \eqref{eq:est-Ek-crack} \tc{the sequence} $\ybk_{\rm tot}$ satisfies the envisioned energy estimate 
\begin{equation*}
 \limsup_{k\goto\infty} kE^{(k)}(\ybk_{\rm tot})\leq E_{\rm lim}(\yb,d_2,d_3)+C\e.
\end{equation*}

It remains to observe that in case \tc{(W9)} holds true with some sequence of rate functions $(M_k)_{k=1}^{\infty}$ and \tc{$||\yb||_\infty \le M$}, then for any $(\zeta_k)_{k=1}^{\infty}\subset(0,1)$ with $\zeta_k \searrow 0$ and $\zeta_k\tc{/M_k} \to \infty$ one can choose \tc{$\ybk_{\rm tot}$} such that $\tg{||\tc{\ybk_{\rm tot}}||_\infty} \le M+\zeta_k$. This is clear by construction for $\ybk$ in \eqref{eq:recov-seq-first} instead of $\ybk_{\rm tot}$ since $\zeta_k \gg \frac{1}{k}$. The bound is indeed preserved by the passage to $\ybk_{\rm tot}$ due to Proposition~\ref{LinftyBdd} once we have $r_kM_kk \ll \zeta_k$. As Proposition~\ref{arbAtomDist} allows us to choose $r_k \searrow 0$ as fast as we wish as long as $r_kk \to \infty$, the claim follows.

Now let us assume that $\yb$, $d_2$, $d_3$ are general as in Theorem \ref{Gamma}(ii). Interestingly, a related approximation problem was treated recently by P. Hornung. \cite{hornung} However, a more elementary construction is sufficient in our case. By a density argument, it is enough to show that there are sequences $(\yb_{\rm tot}^{(j)})_{j=1}^\infty$, $(d_s^{(j)})_{j=1}^\infty$, $s=2,3$, such that:
\begin{enumerate}[(i)]
\item for every $j$ and all $i\in\{1,2,\dots,\bar{n}_{\rm f}+1\}$, the functions satisfy $\yb_{\rm tot}^{(j)}|_{(\sigma^{i-1},\sigma^i)}\in\mathcal{C}^3((\sigma^{i-1},\sigma^i);\R^3)$, $d_2^{(j)}|_{(\sigma^{i-1},\sigma^i)},d_3^{(j)}|_{(\sigma^{i-1},\sigma^i)}\in \mathcal{C}^2((\sigma^{i-1},\sigma^i);\R^3)$ with $R_{\rm tot}^{(j)}=(\pl_{\xb_1}\yb_{\rm tot}^{(j)}|d_2^{(j)}|d_3^{(j)})$ constant on $(\sigma^i-\eta_j,\sigma^i)$ and on $(\sigma^i,\sigma^i+\eta_j)$, $\eta_j>0$, and $(\yb_{\rm tot}^{(j)},d_2^{(j)},d_3^{(j)})\in\calA$; 
\item $\yb_{\rm tot}^{(j)}\goto\yb$ in $L^2(\tc{(0,L)};\R^3)$, $R_{\rm tot}^{(j)}\goto R=(\pl_{\xb_1}\yb|d_2|d_3)$ in $H^1((\sigma^{i-1},\sigma^i);\R^{3\times 3})$ for any $i\in\{1,\dots,\bar{n}_{\rm f}+1\}$;
\item $E_{\rm lim}(\yb_{\rm tot}^{(j)},d_2^{(j)},d_3^{(j)})\goto E_{\rm lim}(\yb,d_2,d_3)$, $j\goto\infty$.
\end{enumerate}

Let $(\eta_j)$ be a positive null sequence. For each $i\in\{1,2,\dots,\bar{n}_{\rm f}+1\}$ we find an approximating sequence $(\tilde{R}^{(j)}|_{(\sigma^{i-1},\sigma^i)})\subset\mathcal{C}^2([\sigma^{i-1},\sigma^i];\R^{3\times 3})$, such that $\tilde{R}^{(j)}$ is constant on $(\sigma^{i-1},\sigma^{i-1}+\eta_j)$ and $(\sigma^{i}-\eta_j,\sigma^{i})$ and $\tilde{R}^{(j)}\goto R$ in $H^1((\sigma^{i-1},\sigma^i);\R^{3\times 3})$ so that $\tilde{R}^{(j)}\goto R$ uniformly in $(\sigma^{i-1},\sigma^i)$ by the Sobolev embedding theorem. Then we project $\tilde{R}^{(j)}(\xb_1)$ for every $\xb_1\in(\sigma^{i-1},\sigma^i)$ smoothly onto ${\rm SO}(3)$ and get a sequence $\{R^{(j)}\}\subset \calC^1([\sigma^{i-1},\sigma^i];\R^{3\times 3})$ of mappings with values in ${\rm SO}(3)$. This implies that $R^{(j)}\goto R$ in $H^1((\sigma^{i-1},\sigma^i);\R^{3\times 3})$ for $i=1,2,\dots,\bar{n}_{\rm f}+1$.

We write $R^{(j)}=(\pl_{\xb_1}\yb^{(j)}|\bar{d}_2^{(j)}|\bar{d}_3^{(j)})$ for $\bar{d}_2^{(j)},\bar{d}_3^{(j)}\in\calC^2([\sigma^{i-1},\sigma^i];\R^3)$ and $\yb^{(j)}\in\calC^3([\sigma^{i-1},\sigma^i];\R^3)$ such that $\yb^{(j)}(\sigma^{i-1}+)=\yb(\sigma^{i-1}+)$; thus we have $(\yb^{(j)}|\bar{d}_2^{(j)}|\bar{d}_3^{(j)})\in\calA$. To avoid issues with crack terms, we rigidly move the pieces of the rod so as to obtain a $j$-independent contribution from the cracks that is exactly equal to the limiting crack energy. We set  
\begin{equation*}
\yb_{\rm tot}^{(j)}(\xb)
=O^{(j,i)}\yb^{(j)}(\xb)+c^{(j,i)} 
\quad\text{and}\quad 
d_s^{(j)}=O^{(j,i)}\bar{d}_s^{(j)}, \;s=2,3,
\end{equation*}
if $\sigma^{i-1} < \xb_1 < \sigma^i,\ i=1,2,\dots,\bar{n}_{\rm f}+1$, where $O^{(j,i)} \in {\rm SO}(3)$ and $c^{(j,i)}\in\R^3$ are defined consecutively by $O^{(j,0)} = \Id$, $c^{(j,0)}=0$, and requiring that 
\begin{equation*}
 \yb_{\rm tot}^{(j)}(\sigma^i+)-\yb_{\rm tot}^{(j)}(\sigma^i-)
 = \yb(\sigma^i+)-\yb(\sigma^i-)
 \quad\text{and}\quad
 [R_{\rm tot}^{(j)}(\sigma^i-)]^{-1}R_{\rm tot}^{(j)}(\sigma^i+)
 =[R(\sigma^i-)]^{-1}R(\sigma^i+)
\end{equation*}
for $i=1,\ldots,\bar{n}_{\rm f}$, $R_{\rm tot}^{(j)}=(\pl_{\xb_1}\yb_{\rm tot}^{(j)}|d_2^{(j)}|d_3^{(j)}),\; j\in\N$. By frame indifference, the elastic energy is not changed by such an operation. Noting that $O^{(j,i)} \to \Id$ and $c^{(j,i)}\to 0$ for $j\goto\infty$, \tc{we see that} these mappings are such that (i)–(iii) hold (for (iii) observe that the integral in \eqref{eq:limsupPartConcl} behaves continuously in $R$ with respect to the topologies used here).
\end{proof}

\section{Examples}

Finally, we list a few examples of mass-spring models treatable by our methods: a model with rather general pair interactions, the so-called truncated and shifted Lennard-Jones potential (LJTS), `truncated harmonic spring', and a simplified highly brittle model.

\begin{example}\label{genPot}
As general nearest-neighbour (NN) and next-to-nearest-neighbour (NNN) interactions on a cubic lattice, we can consider
\begin{equation}\label{eq:EkPair}
E^{(k)}(y)=\frac{1}{2}\sum_{\substack{\hat{x}_*,\hat{x}_{**}\in\hat{\Lambda}_k\\ |\hat{x}_*-\hat{x}_{**}|=1}}W_{\rm NN}^{(k)}(|\hat{y}(\hat{x}_*)-\hat{y}(\hat{x}_{**})|)+\frac{1}{2}\sum_{\substack{\hat{x}_*,\hat{x}_{**}\in\hat{\Lambda}_k\\ |\hat{x}_*-\hat{x}_{**}|=\sqrt{2}}}W_{\rm NNN}^{(k)}\bigl(\frac{|\hat{y}(\hat{x}_*)-\hat{y}(\hat{x}_{**})|}{\sqrt{2}}\bigr)\tb{+\mathcal{X}_k(y)},
\end{equation}
where $y\colon\Lambda_k\goto\R^3$, $\hat{y}(\hat{x})=\tb{ky(\frac{1}{k}\hat{x})}$, $\hat{x}\in\hat{\Lambda}_k$, and $W_{\rm NN}^{(k)}$, $W_{\rm NNN}^{(k)}$ satisfy the following list of assumptions:
\begin{enumerate}
\item[(P1)] $W_{\rm NN(N)}^{(k)}\colon [0,\infty)\goto [0,\infty]$ is continuous and finite on $(0,\infty)$ and $W_{\rm NN(N)}^{(k)}(r)=0$ if and only if $r=1$;
\item[(P2)] there is a sequence $(c_{\rm f}^{(k)})_{k=1}^\infty$ with $c_{\rm f}^{(k)}\searrow 0$ and $\lim_{k\goto\infty}k[c_{\rm f}^{(k)}]^2\in (0,\infty)$ such that
$$W_{\rm NN(N)}^{(k)}(r)=W_{\rm 0NN(N)}(r)$$
for all $r\in(1-c_{\rm f}^{(k)},1+c_{\rm f}^{(k)})$, where $W_{\rm 0NN(N)}$ is of class $\mathcal{C}^2$ and $W_{0NN(N)}''(1)>0$;
\item[(P3)] $W_{\rm NN(N)}^{(k)}(r)=\bar{W}_{\rm NN(N)}^{(k)}(r)$ if $r\in[0,1-c_{\rm f}^{(k)}]\cup[1+c_{\rm f}^{(k)},\infty)$; the function $\bar{W}_{\rm NN(N)}^{(k)}$ is bounded from below by $\bar{c}_{\rm NN(N)}^{(k)}$ such that $k\bar{c}_{\rm NN(N)}^{(k)}\goto \bar{c}_{\rm NN(N)}>0$ and $(k+1)W_{\rm NN(N)}^{(k+1)}\geq kW_{\rm NN(N)}^{(k)}$ for every $k\in\N$;
\item[(P4)]\tg{$\bar{W}_{\rm NN(N)}^{(k)}(r)=\omega_{\rm NN(N)}^{(k)}+\frac{1}{k}\mathscr{r}_{\rm NN(N)}(r)$ if $r\geq k\bar{M}_k$ for $\bar{M}_k\goto 0$ with $k\bar{M}_k\goto\infty$, $\mathscr{r}_{\rm NN(N)}(r)=O(r^{-1})$, $r\goto\infty$, and $\lim_{k\goto\infty}k\omega_{\rm NN(N)}^{(k)}\in (0,\infty)$.}
\end{enumerate}
\tb{To guarantee preservation of orientation, in \eqref{eq:EkPair} we have included a nonnegative term $\mathcal{X}_k(y)$ that gives rise to $\chi^{(k)}$ below.} \tg{Thus $E^{(k)}$ can be written in the form \eqref{eq:Ek} as a sum of cell energies with}
\begin{equation}\label{eq:genPot}
W_{\rm cell}^{(k)}(\vec{y})=\frac{1}{8}\sum_{|\zf^i-\zf^j|=1}W_{\rm NN}^{(k)}(|\hat{y}_i-\hat{y}_j|)+\frac{1}{4}\sum_{|\zf^i-\zf^j|=\sqrt{2}}W_{\rm NNN}^{(k)}\Bigl(\frac{|\hat{y}_i-\hat{y}_j|}{\sqrt{2}}\Bigr)+\tg{\chi^{(k)}(\vec{y})}
\end{equation}
for $\vec{y}=(\hat{y}_1|\cdots|\hat{y}_8)\in\R^{3\times 8}$ and \tg{the functions $W_{\rm surf}^{(k)}$, $W_{\rm end}^{(k)}$ constructed in a similar manner to account for surface contributions to atomic bonds lying on the rod's boundary (see \cite[Subsection~2.4]{elRods}).} The frame-indifferent term $\chi^{(k)}$, $C/k\geq\chi^{(k)}\geq 0$, penalizes deformations that are not locally orientation-preserving, i.e.\ it is greater than or equal to \tb{$\bar{c}/k$, $\bar{c}>0$,} on a \tb{$k$-independent} neighbourhood of ${\rm O}(3)\bar{\Id}\setminus \bar{\rm SO}(3)$ and vanishes \tb{otherwise} (see \cite{BS06,FrdS15}). \tb{An alternative to penalties such as $\mathcal{X}_k$ and $\chi^{(k)}$ is cell energies with ${\rm O}(3)$-invariance, see \cite[Section~2.4]{BrS19}.}

It can be shown that potentials \tg{$W_{\rm NN}^{(k)}$, $W_{\rm NNN}^{(k)}$} as above make the corresponding $W_{\rm cell}^{(k)}$ admissible, i.e. \tc{(W1)--(W6)}, and \tc{(W9)} hold (\tc{(W9) is a consequence of (P4)}). In particular, the \textit{truncated and splined Lennard-Jones potential} from \cite{HEvans} and versions thereof fall under this case, with appropriately chosen parameters.
\end{example}

\begin{example}
Let
\begin{equation*}
W_{\rm LJ}(r)=d\left(\frac{1}{r^{12}}-\frac{2}{r^6}\right)+d,
\end{equation*}
where $r\in (0,\infty)$ and $d>0$ is a parameter (note that $\lim_{r\goto\infty}W_{\rm LJ}(r)=d$ and $\argmin_{r>0}W_{\rm LJ}(r)=1$). Further we set
\begin{equation*}
W_{\rm LJTS}^{(k)}(r)=\begin{cases}
W_{\rm LJ}(r) & r\in (0,1)\\
\min\{W_{\rm LJ}(r),\frac{1}{k}\} & r\in [1,\infty)
\end{cases}.
\end{equation*}
We again consider pair interactions, so the cell energy function takes the form \eqref{eq:genPot} with $W_{\rm LJTS}^{(k)}$ in place of $W_{\rm NN}^{(k)}$ and $W_{\rm NNN}^{(k)}$. The property $(k+1)W_{\rm cell}^{(k+1)}\geq kW_{\rm cell}^{(k)}$ can be proved by discussing for each bond if it is deformed elastically or if the truncation is active. Computing the value of $r$ beyond which truncation applies in $W_{\rm LJTS}^{(k)}$, we observe that assumptions \tc{(W3)} and \tc{(W5)} hold with $c_{\rm frac}^{(k)}=[\sqrt[6]{d+\sqrt{d/k}}-\sqrt[6]{d-(1/k)}]/(2\sqrt[6]{d-(1/k)})$ and $W_0$ being the sum of Lennard-Jones interactions with no truncation. By the properties of $\nabla^2 W_0(\bar{\Id})$, the estimate $\hat{C}W_0(\vec{y})\geq\mathrm{dist}^2(\bar{\nabla}\hat{y},{\rm SO}(3))$ holds with a constant $\hat{C}>0$ and the usual symbol $\bar{\nabla}\hat{y}$ denoting the discrete gradient of $\vec{y}\in\R^{3\times 8}$ (cf.~\cite[Lemma~3.2~and~Section~7]{BS06}).

Moreover, we claim that if $\mathrm{dist}(\bar{\nabla}\hat{y}, \bar{\rm SO}(3))>c_{\rm frac}^{(k)}$, then $W_{\rm cell}^{(k)}(\vec{y})\geq\min\{1/(8k),[c_{\rm frac}^{(k)}]^2/\hat{C}\}=:\bar{c}_1^{(k)}$. Indeed, as long as $W_{\rm cell}^{(k)}(\vec{y})<\bar{c}_1^{(k)}$, the cutoff is not active in any interatomic bond (the arguments of $W_{\rm LJTS}^{(k)}$ are close enough to $1$) and thus $W_{\rm cell}^{(k)}(\vec{y})=W_0(\vec{y})$ so that $\mathrm{dist}(\bar{\nabla}\hat{y}, \bar{\rm SO}(3))\leq c_{\rm frac}^{(k)}$.  This shows the second part of assumption \tc{(W5)}.
\end{example}

\begin{example}
For the functions
\begin{equation*}
W_{\rm harm}(r)=K(r-1)^2,\quad W_{\rm TH}^{(k)}(r)=\begin{cases}
\min\{W_{\rm harm}(r),\frac{c_{\rm TH}^+}{k}\} & r\geq 1\\
\min\{W_{\rm harm}(r),\frac{c_{\rm TH}^-}{k}\} & r< 1
\end{cases}
\end{equation*}
with positive constants $K$, $c_{\rm TH}^-$, $c_{\rm TH}^-$, one can similarly find $c_{\rm frac}^{(k)}$ and $\bar{c}_1^{(k)}$ so that $W_{\rm cell}^{(k)}$ defined by \eqref{eq:genPot} with $W_{\rm NN}^{(k)}$ and $W_{\rm NNN}^{(k)}$ replaced by $W_{\rm TH}^{(k)}$ is an admissible cell energy.
\end{example}

\begin{example}
Another simplified model can be obtained if we set
\begin{equation*}
W_{\rm cell}^{(k)}(\vec{y})=\min\{W_0(\vec{y}),\bar{c}_1^{(k)}\}
\end{equation*}
and $c_{\rm W}$, $\bar{c}_1^{(k)}$, and frame-indifferent $W_0$ are as in assumptions \tc{(W3), (W5)}. This corresponds to $\bar{W}^{(k)}\equiv\bar{c}_1^{(k)}$ and the cell formula then reduces to $\ph(u,R)\equiv (\sharp\mathcal{L}')\tg{c_{\rm W}}\bar{c}_1$, where $\bar{c}_1=\lim_{k\goto\infty}k\bar{c}_1^{(k)}$, for any $u\in\R^3$ and $R\in{\rm SO}(3)$ except $(u,R)=(0,\Id)$ (specifically, we use sublevel sets of $W_0$ instead of $\mathrm{dist}^2(\bar{\nabla}\hat{y},\bar{\rm SO}(3))$ to define the threshold distinguishing between $W_0$ and $\bar{W}^{(k)}$, but our findings remain valid in this case as well).
\end{example}

\section{Explicit calculation of crack energy}

For mass-spring models, it is possible to simplify further \eqref{eq:phi} in specific situations.
\begin{prop}\label{explCrack}
If \tg{$E^{(k)}$ is given by \eqref{eq:EkPair}} and assumptions \textup{(P1)--(P4)} hold, together with
\begin{enumerate}
\item[$\mathrm{(P5)}$] $\lim_{k\goto\infty}k\bar{W}_{\rm NN(N)}^{(k)}(r_k)=\omega_{\rm NN(N)}$ for any sequence $r_k\goto\infty$,
\end{enumerate}
for $W_{\rm NN}^{(k)}$ and $W_{\rm NNN}^{(k)}$, then
\begin{equation*}
\ph(u,R)=(\sharp\mathcal{L})\omega_{\rm NN}+\sharp\{(\xb',\xb_*')\in\mathcal{L}^2;\; |\xb'-\xb_*'|=1\}\omega_{\rm NNN}
\end{equation*}
for any $0\neq u\in\R^3$ and $R\in{\rm SO}(3)$.
\end{prop}
\begin{proof}
\textit{Step }1. The mapping $\mathscr{v}^{(k)}$ defined as
\begin{gather*}
\mathscr{v}^{(k)}(w_1,\xb')=\begin{cases}
R_-^{(k)}\T{(r_kw_1,\frac{1}{k}\xb')}+y_-^{(k)} & \text{ on }[-1,0]\times S^{\rm ext}\\
R_+^{(k)}\T{(r_kw_1,\frac{1}{k}\xb')}+y_+^{(k)} & \text{ on }[r_k^{-1}k^{-1},1]\times S^{\rm ext},
\end{cases}\\
R_\pm^{(k)}\in {\rm SO}(3),\; y_\pm^{(k)}\in\R^3,\; (R_-^{(k)})^{-1}R_+^{(k)}\goto R,\; y_+^{(k)}-y_-^{(k)}\goto u;\; r_k^{-1}\goto\infty\text{ as }o(k),
\end{gather*}
and interpolated to be piecewise affine ($\mathscr{v}^{(k)}\in\mathrm{PAff}(\tc{\Lambda_{r_k,k}})$) has the property that
\begin{equation*}
\lim_{k\goto\infty}\calE_k(\mathscr{v}^{(k)},[-1,1])=(\sharp\mathcal{L})\omega_{\rm NN}+\sharp\{(\xb',\xb_*')\in\mathcal{L}^2;\; |\xb'-\xb_*'|=1\}\omega_{\rm NNN}.
\end{equation*}
Thus we find that $\ph(u,R)$ is less than or equal to the right-hand side in the above equation.

\textit{Step }2. Given $\e>0$, we find sequences $((r_k)_{k=1}^\infty,(k)_{k=1}^\infty,(\mathscr{y}^{(k)})_{k=1}^\infty)\in\mathcal{V}_{u,R}$ such that
\begin{equation}\label{eq:msPhiApprox}
\limsup_{k\goto\infty}\calE_k(\mathscr{y}^{(k)},[-1,1])\leq \ph(u,R)+\e,
\end{equation}
using Proposition \ref{arbAtomDist}. Set
\tb{\begin{equation*}
\bar{W}_1^{(k)}:=\frac{1}{r_kk}\Bigl\{\bigl\lfloor-r_kk\bigl\rfloor+\frac{3}{2},\bigl\lfloor-r_kk\bigl\rfloor+\frac{5}{2},\dots,\bigl\lfloor r_kk\bigl\rfloor-\frac{1}{2}\Bigr\}.
\end{equation*}}
We show that the nature of our pair interactions causes at least one large gap in the spacing of atoms within each fibre which the rod consists of.

\textit{Claim} 1: For each $\xb'\in\mathcal{L}$ and every $T>1$ there is a $k_0\in\N$ such that whenever $k\geq k_0$, we can find some $\bar{w}_1\in \tb{\bar{W}_1^{(k)}}$ satisfying
\begin{equation*}
\frac{|\mathscr{y}^{(k)}(\bar{w}_1+\frac{1}{2r_kk},\xb')-\mathscr{y}^{(k)}(\bar{w}_1-\frac{1}{2r_kk},\xb')|}{1/k}>T.
\end{equation*}
 \textit{Proof of claim: }If the converse were true, there would be a $\tilde{T}>1$ and an increasing sequence $\{k_n\}_{n=1}^\infty\subset\N$ such that for all $\bar{w}_1\in\tb{\bar{W}_1^{(k_n)}}$:
\begin{equation*}
k_n|\mathscr{y}^{(k_n)}(\bar{w}_1+\frac{1}{2r_{k_n}k_n},\xb')-\mathscr{y}^{(k_n)}(\bar{w}_1-\frac{1}{2r_{k_n}k_n},\xb')|\leq \tilde{T}.
\end{equation*}
\tb{Then we would get
\begin{multline*}
0\neq |u|=\bigl|\mathscr{y}^{(k_n)}(\max \bar{W}_1^{(k_n)}+\frac{1}{2r_{k_n}k_n},\xb')-\mathscr{y}^{(k_n)}(\min \bar{W}_1^{(k_n)}-\frac{1}{2r_{k_n}k_n},\xb')\bigr|+o_{n\goto\infty}(1)\\
\leq \sum_{\bar{w}_1\in \bar{W}_1^{(k_n)}}\bigl|\mathscr{y}^{(k_n)}(\bar{w}_1+\frac{1}{2r_{k_n}k_n},\xb')-\mathscr{y}^{(k_n)}(\bar{w}_1-\frac{1}{2r_{k_n}k_n},\xb')\bigr|+o_{n\goto\infty}(1)\leq 2r_{k_n}\frac{k_n}{k_n}\tilde{T}+o_{n\goto\infty}(1)\goto 0,
\end{multline*}
which is a contradiction.}

\textit{Step }3. A similar argument applies to NNN bonds (`diagonal springs') -- if we use zigzag chains of atoms instead of straight atomic fibres. We state the corresponding claim without proof.

\textit{Claim} $\sqrt{2}$: For each $(\xb',\xb_*')\in\mathcal{L}\times\mathcal{L}$ with $|\xb_*'-\xb'|=1$ and every $T>1$ there is a $k_0\in\N$ such that whenever $k\geq k_0$, we can find a $j\in\N$ and $\bar{w}_1=\tb{\frac{1}{r_kk}(\lfloor -r_kk\rfloor+\frac{2j+1}{2})\in \bar{W}_1^{(k)}}$ such that $\mathscr{y}^{(k)}$ from \eqref{eq:msPhiApprox} satisfies:
\begin{equation*}
\frac{|\mathscr{y}^{(k)}(\bar{w}_1+(-1)^{j+1}\frac{1}{2r_kk},\xb_*')-\mathscr{y}^{(k)}(\bar{w}_1+(-1)^j\frac{1}{2r_kk},\xb')|}{\sqrt{2}/k}>T.
\end{equation*}

\textit{Step }4. Since Claims 1 and $\sqrt{2}$ hold for every approximating sequence $(\mathscr{y}^{(k)})_{k=1}^\infty$ fulfilling \eqref{eq:msPhiApprox}, we get
\begin{equation*}
(\sharp\mathcal{L})\omega_{\rm NN}+\sharp\{(\xb',\xb_*')\in\mathcal{L}^2;\; |\xb'-\xb_*'|=1\}\omega_{\rm NNN}\leq\ph(u,R)+\e.
\end{equation*}
As this is valid for any $\e>0$, the desired conclusion follows.
\end{proof}

\begin{prop}
Under the assumptions of Proposition \ref{explCrack} and further supposing
\begin{enumerate}
\item[$\mathrm{(P6)}$] $W_{\rm NN}^{(k)}$, $W_{\rm NNN}^{(k)}$ are nondecreasing on $[1,\infty)$,
\end{enumerate}
we have
\begin{equation*}
0<\ph(0,R)<\ph(u,\tilde{R})
\end{equation*}
for any $R,\tilde{R}\in{\rm SO}(3)$, $R\neq\Id$ and $0\neq u\in\R^3$.
\end{prop}

\begin{proof}
The first inequality was shown in Remark \ref{phBdd}.

As to the second inequality, Proposition \ref{explCrack} implies that for a nonzero $u$, the crack energy $\ph(u,R)$ is independent of $R$, hence we limit ourselves to the case $\tilde{R}=R$ without loss of generality. If $R\in{\rm SO}(3)$ and $u\in\R^3\setminus\{0\}$ are fixed, it is enough to find a sequence $(\mathscr{v}_0^{(k)})_{k=1}^\infty$ of deformations admissible in the definition of $\ph(0,R)$ such that
\begin{equation*}
\limsup_{k\goto\infty}\calE_{k}(\mathscr{v}_0^{(k)};[-1,1])<(\sharp\mathcal{L})\omega_{\rm NN}+\sharp\{(\xb',\xb_*')\in\mathcal{L}^2;\; |\xb'-\xb_*'|=1\}\omega_{\rm NNN}
\end{equation*}
by Proposition \ref{explCrack}. Fix $k\in\N$ and let $\mathscr{v}^{(k)}$, $R_\pm^{(k)}$, $r_k$, and $y_\pm^{(k)}$ be as in the proof of Proposition \ref{explCrack} with our new definitions of $R$ and $u$. We define 
\begin{align*}
 F^\pm:=\Bigl\{R_\pm^{(k)}\T{\bigl(\frac{1}{2k}\pm\frac{1}{2k},\frac{1}{k}\xb'\bigr)}+y_\pm^{(k)};\; \xb'\in\mathcal{L}\Bigr\}
\end{align*}
 and observe that ${\rm dist}(F^+,F^-) = |y^{(k)}_+-y^{(k)}_-|+O(\tg{\frac{1}{k}}) = |u| + \tg{o_{k\goto\infty}(1)}$. Now we choose $\xb'_0\in\mathcal{L}$ and consider configurations with shifted right parts, given by 
\begin{align*}
 \mathscr{v}^{(k)}(w_1, \xb'; t) 
=\begin{cases}
R_-^{(k)}\T{(r_kw_1,\frac{1}{k}\xb')}+y_-^{(k)} & \text{ on }[-1,0]\times S^{\rm ext}\\
R_+^{(k)}\T{(r_kw_1,\frac{1}{k}\xb')}+y_+^{(k)}-c_0^{(k)}(t) & \text{ on }[r_k^{-1}k^{-1},1]\times S^{\rm ext},
\end{cases} 
\end{align*}
where $c_0^{(k)}(t) = t[\mathscr{v}^{(k)}(\frac{1}{r_kk},\xb'_0) - \mathscr{v}^{(k)}(0,\xb'_0)]$, $t \in [0,1]$. 
We then define $t_0^{(k)}$ to be the smallest $t \in [0,1]$ such that 
\begin{align*}
\bigl|\mathscr{v}^{(k)}\bigl(\frac{1}{r_kk},\xb';t\bigr) - \mathscr{v}^{(k)}(0,\xb';t)\bigr| = \frac{1}{k} 
\qquad\text{or}\qquad
\bigl|\mathscr{v}^{(k)}\bigl(\frac{1}{r_kk},\xb_*';t\bigr) - \mathscr{v}^{(k)}(0,\xb_{**}';t)\bigr| = \frac{\sqrt{2}}{k}
\end{align*}
for some $\xb'\in\mathcal{L}$, \tg{or else}, $\xb_*',\xb_{**}'\in\mathcal{L}$ with $|\xb_*'-\xb_{**}'| = 1$, \tg{respectively}.  By construction such $t_0^{(k)} \in (0,1)$ exists if $k$ is large enough  and we have $|c_0^{(k)}(\tg{t_0^{(k)}}) - u|\tg{\goto 0}$ as $k\goto\infty$. Setting $\mathscr{v}^{(k)}_0 = \mathscr{v}^{(k)}(\;\cdot\; ; t_0^{(k)})$ and recalling $\mathrm{(P6)}$ we find
\begin{equation}\label{eq:vlessE}
\calE_{k}(\mathscr{v}_0^{(k)};[-1,1])\le(\sharp\mathcal{L})\omega_{\rm NN}+\sharp\{(\xb',\xb_*')\in\mathcal{L}^2;\; |\xb'-\xb_*'|=1\}\omega_{\rm NNN}-\min\{\omega_{\rm NN},\omega_{\rm NNN}\}.
\end{equation}
We still need to check that the sequence $(\mathscr{v}_0^{(k)})_{k=1}^\infty$ thus constructed satisfies the correct boundary conditions for $\ph(0,R)$. But this is clear, since $|y_+^{(k)}-c_0^{(k)}(t_0^{(k)})-y_-^{(k)}|\tg{\goto 0}$.
\end{proof}

\section{Discussion}

Our work makes a contribution to the modelling of elastic-brittle ultrathin structures, but as such, it could be certainly extended in various directions. 

We remark that the situation becomes considerably more difficult for plates due to a much richer phenomenology of crack and kink patterns. For bending-dominated configurations also severe geometric obstructions that result from the isometry constraints are encountered. A first step has recently been achieved in \cite{BZK}, where a `Blake--Zisserman--Kirchhoff theory' has been derived for plates with soft inclusions.

From the point of view of applications, it would be interesting to extend our findings to other crystallographic lattices (such as diamond cubic as in \cite{LazSchlo2} or zincblende), heterogeneous nanostructures with several different types of atoms, or to study the influence of lattice defects.

The model could also be studied computationally (e.g. numerical approximations of the cell formula could be implemented).

\section*{Acknowledgements}
The authors acknowledge the support of the Deutsche Forschungsgemeinschaft (DFG, German Research Foundation) within the Priority Programme SPP 2256 `Variational Methods for Predicting Complex Phenomena in Engineering Structures and Materials'.

\bibliographystyle{alpha} 
\renewcommand{\bibname}{References}
\bibliography{bibliography}
\end{document}